\documentclass[12pt]{amsart}

 \usepackage{amsfonts,graphics,amsmath,amsthm,amsfonts,amscd,amssymb,amsmath,latexsym,multicol}
\usepackage{epsfig,url}
\usepackage{flafter}

\makeatletter

\def\jobis#1{FF\fi
  \def\predicate{#1}%
  \edef\predicate{\expandafter\strip@prefix\meaning\predicate}%
  \edef\job{\jobname}%
  \ifx\job\predicate
}

\makeatother

\if\jobis{proposal}%
\else
\fi

 \usepackage[matrix, arrow]{xy}

\DeclareMathOperator{\Supp}{Supp}

\DeclareMathOperator{\num}{num}

\DeclareMathOperator{\NS_R}{\rm NS_{\R}}

 \newcommand{\N}{\mathbb N}
 
 \newcommand{\Q}{\mathbb Q}
 \newcommand{\R}{\mathbb R}
 \newcommand{\Z}{\mathbb Z}
 \newcommand{\bir}{\dashrightarrow}
 \newcommand{\rddown}[1]{\left\lfloor{#1}\right\rfloor} % round-down
 \newcommand{\wt}{\widetilde}
 \newcommand{\wh}{\widehat}
 
 \newcommand{\cI}{\mathcal I}

 \numberwithin{equation}{subsection}
 \numberwithin{footnote}{subsection}

 \newtheorem{cor}[subsection]{Corollary}
 \newtheorem{lem}[subsection]{Lemma}
 \newtheorem{prop}[subsection]{Proposition}
 \newtheorem{thm}[subsection]{Theorem}

\newtheorem{claim}[subsection]{Claim}
{
 \newtheorem{defn}[subsection]{Definition}

 \newtheorem{rem}[subsection]{Remark}

}
 \newenvironment{pfof}[1]{\paragraph{Proof of #1}}{\par\medskip}

 \newcommand{\ke}[1]{$\acute{\mbox{e}}$}
 \newcommand{\ku}[1]{$\acute{\mbox{u}$}}
 \newcommand{\kl}[1]{$\acute{\mbox{l}}$}
 \newcommand{\kh}[1]{$\acute{\mbox{h}}$}
 \newcommand{\kr}[1]{$\acute{\mbox{r}}$}
 \newcommand{\kx}[1]{$\acute{\mbox{x}}$}
 \newcommand{\ki}[1]{${\^\i}$}

\title[Minimal models, flips...]{Minimal models, flips and \\ finite generation : \\ a tribute to 
V.V. SHOKUROV and Y.-T. SIU}
\author{Caucher Birkar \hspace{2cm}  Mihai P\u{a}un}
\date{\today}
\begin{document}
\maketitle %{Minimal models,  flips ...}

\begin{abstract}
In this paper, we discuss a proof of existence of log 
minimal models or Mori fibre spaces for klt pairs $(X/Z,B)$ with $B$ big$/Z$. This then implies existence of klt log flips, 
finite generation of klt log canonical rings, and most of the other results of 
 Birkar-Cascini-Hacon-McKernan paper [\ref{BCHM}].  
\end{abstract}

%\tableofcontents
%%%%%%%%%%%%%%%%%%%%%%%%

\section{Introduction}

We consider pairs $(X/Z,B)$ where $B$ is an $\R$-boundary and $X\to Z$ is a projective morphism of normal 
quasi-projective varieties over an algebraically closed field $k$ of characteristic zero. We call a pair $(X/Z,B)$ 
effective if there is an $\R$-divisor $M\ge 0$ such that $K_X+B\equiv M/Z$.

\begin{thm}\label{main}
Let $(X/Z,B)$ be a klt pair where $B$ is big$/Z$. 
Then,\\\\ 
$\rm (1)$ if $K_X+B$ is pseudo-effective$/Z$, then $(X/Z,B)$ has a log minimal model,\\
$\rm (2)$ if $K_X+B$ is not pseudo-effective$/Z$, then $(X/Z,B)$ has a Mori fibre space.
\end{thm}

\begin{cor}[Log flips]\label{c-flips}
Log flips exist for klt (hence $\Q$-factorial dlt) pairs.
\end{cor}

\begin{cor}[Finite generation]\label{c-fg}
Let $(X/Z,B)$ be a klt pair where $B$ is a $\Q$-divisor and $f\colon X\to Z$ the given morphism. Then, the log canonical 
sheaf 
$$
\mathcal{R}(X/Z,B):=\bigoplus_{m\ge 0} f_*\mathcal{O}_X(m(K_X+B)) 
$$ 
is a finitely generated $\mathcal{O}_Z$-algebra.
\end{cor}

The proof of the above theorem is divided into two independent parts. First we have 

\begin{thm}[Log minimal models]\label{t-mmodels}
Assume $\rm (1)$ of Theorem \ref{main} in dimension $d-1$ and let $(X/Z,B)$ be a klt pair 
of dimension $d$ where $B$ is big$/Z$. If $(X/Z,B)$ is effective, 
then it has a log minimal model.
\end{thm}

A proof of this theorem is given in section 2 based on ideas in [\ref{B-mmodels}][\ref{BCHM}]. Second we have 

\begin{thm}[Nonvanishing]\label{t-nonvanishing}
Let $(X/Z,B)$ be a klt pair where $B$ is big$/Z$. 
If $K_X+B$ is pseudo-effective$/Z$, then $(X/Z,B)$ is effective.
\end{thm}

A proof of this theorem is given in section 3; it is mainly based on the arguments in 
[\ref{mp2}] (see [\ref{shok}] and [\ref{siu2}] as well). We give a rough comparison of the 
proof of this theorem with the proof of the corresponding theorem in [\ref{BCHM}], that is [\ref{BCHM}, Theorem D]. 

One crucial feature of the proof of Theorem \ref{t-nonvanishing} is that it does not use the log minimal model program. 
The proof goes as follows: 

(a) We first assume that $Z$ is a point, and using Zariski type decompositions one can 
create lc centres and pass to plt pairs, more precisely, by going on a sufficiently high resolution we can replace $(X/Z,B)$ by a plt 
pair $(X/Z,B+S)$ where $S$ is a smooth prime divisor, $K_X+B+S|_S$ is pseudo-effective, 
$B$ is big and its components do not intersect. 

(b) By induction, the ${\mathbb R}$-bundle $K_X+B+S|_S$ has an ${\mathbb R}$-section say $T$, which can be assumed to be singular enough for the extension purposes.

(c) Diophantine approximation of the couple $(B, T)$: we can find pairs $(B_i, T_i)$ with rational coefficients and sufficiently close to 
$(B, T)$ (in a very precise sense) such that 
$$K_X+B+S =\sum r_i(K_X+B_i+S)$$ 
for certain real numbers $r_i\in [0,1]$ and such that all the pairs $(X/Z,B_i+S)$ 
are plt and each $K_X+B_i+S|_S$ is numerically equivalent with $T_i$.
Moreover, one can improve this to  $K_X+B_i+S|_S\sim_{\Q} T_i\ge 0$. 

(d) Using the {\sl invariance of plurigenera techniques}, 
one can lift this to  $K_X+B_i+S\sim_{\Q} M_i\ge 0$ 
and then a relation $K_X+B+S\equiv M\ge 0$. 

(e) Finally, we get the theorem in the general case, i.e. when $Z$ is not a point, using positivity properties of 
direct image sheaves and another application of extension theorems.
\medskip

In contrast, the log minimal model  program  is an important 
ingredient of the proof of [\ref{BCHM}, Theorem D] which proceeds as follows: 
(a') This step is the same as (a) above. 
(b') By running the log minimal model  program appropriately and using induction 
on finiteness of log minimal models and termination with scaling, one constructs a model $Y$ birational 
to $X$ such that $K_Y+B_Y+S_Y|_{S_Y}$ is nef where $K_Y+B_Y+S_Y$ is the push down of $K_X+B+S$. 
Moreover, by Diophantine approximation, we can find boundaries $B_i$ with rational coefficients and sufficiently close to 
$B$  such that $B=\sum r_iB_i$ for certain real numbers $r_i\in [0,1]$ and such that each   $K_Y+B_{i,Y}+S_Y$ 
is plt and $K_Y+B_{i,Y}+S_Y|_{S_Y}$ is nef.
(c') By applying induction and the base point free theorem one gets $K_Y+B_{i,Y}+S_Y|_{S_Y}\sim_{\Q} N_i\ge 0$.
(d') The Kawamata-Viehweg vanishing theorem now gives $K_Y+B_{i,Y}+S_Y\sim_{\Q} M_{i,Y}\ge 0$ from which 
we easily get a relation $K_X+B+S\sim_{\R} M\ge 0$. 
(e') Finally, we get the theorem in the general case, i.e. when $Z$ is not a point,  by restricting to the generic 
fibre and applying induction.

%%%%%%%%%%%%%%%%%%%%%%%%%%%%%%%%%%%%%%%%%%%%%%%%%

\section{Log minimal models}

In this section we prove Theorem \ref{t-mmodels} (cf. [\ref{BCHM}, Theorems A, B, C, E]). The results 
in this section are also implicitly or explicitly proved in [\ref{BCHM}]. 
We hope that this section also helps the reader to read [\ref{BCHM}].

\subsection*{Preliminaries}
Let $k$ be an algebraically closed field of characteristic zero fixed throughout this 
section. When we write an $\R$-divisor $D$ as $D=\sum d_iD_i$ (or similar notation) we mean that $D_i$ are distinct prime divisors. The norm $||D||$ is defined as 
$\max\{|d_i|\}$. For a birational map $\phi\colon X\bir Y$ 
and an $\R$-divisor $D$ on $X$ we often use $D_Y$ to mean  the birational transform of 
$D$, unless specified otherwise.

A \emph{pair} $(X/Z,B)$ consists of normal quasi-projective varieties $X,Z$ over $k$, an $\R$-divisor $B$ on $X$ with
coefficients in $[0,1]$ such that $K_X+B$ is $\mathbb{R}$-Cartier, and a projective 
morphism $X\to Z$. For a prime divisor $E$ on some birational model of $X$ with a
nonempty centre on $X$, $a(E,X,B)$ denotes the log discrepancy.

 An $\R$-divisor $D$ on $X$ is called \emph{pseudo-effective}$/Z$ if up to numerical equivalence/$Z$ it is the limit of effective $\R$-divisors, i.e. for any ample$/Z$ $\R$-divisor $A$ and real number $a>0$, $D+aA$ is big$/Z$. A pair $(X/Z,B)$ is called \emph{effective} if there is an $\R$-divisor $M\ge 0$ such that $K_X+B\equiv M/Z$;
in this case, we call $(X/Z,B,M)$ a \emph{triple}.  By a log resolution of a triple $(X/Z,B,M)$
we mean a log resolution of $(X, \Supp B+M)$. A triple $(X/Z,B,M)$ is log smooth if 
$(X, \Supp B+M)$ is log smooth. When we refer to a triple as being lc, dlt, etc, we mean that 
the underlying pair $(X/Z,B)$ has such properties. 

For a triple  $(X/Z,B,M)$,
define
$$
\theta(X/Z,B,M):=\#\{i~|~m_i\neq 0 ~~\mbox{and}~~ b_i\neq 1\}
$$
where $B=\sum b_iD_i$ and $M=\sum m_iD_i$.

Let $(X/Z,B)$ be a lc pair. By a \emph{log flip}$/Z$ we mean the flip of a $K_X+B$-negative extremal flipping contraction$/Z$, 
and by a \emph{pl flip}$/Z$ we mean a log flip$/Z$ when $(X/Z,B)$ is $\Q$-factorial dlt and the log flip is also an $S$-flip for 
some component $S$ of $\rddown{B}$, i.e. $S$ is numerically negative on the flipping contraction.

A \emph{sequence of log flips$/Z$ starting with} $(X/Z,B)$ is a sequence $X_i\bir X_{i+1}/Z_i$ in which  
$X_i\to Z_i \leftarrow X_{i+1}$ is a $K_{X_i}+B_i$-flip$/Z$, $B_i$ is the birational transform 
of $B_1$ on $X_1$, and $(X_1/Z,B_1)=(X/Z,B)$.

\begin{defn}[Log minimal models and Mori fibre spaces]\label{d-mmodel}
Let $(X/Z,B)$ be a dlt pair, $(Y/Z,B_Y)$ a $\Q$-factorial dlt pair, $\phi\colon X\bir Y/Z$ a birational map such that $\phi^{-1}$ does not contract divisors, and $B_Y=\phi_*B$. 

$\rm (1)$ We say that $(Y/Z,B_Y)$ is a nef model of $(X/Z,B)$ if $K_Y+B_Y$ is nef$/Z$. 
We say that $(Y/Z,B_Y)$ is a log minimal model of $(X/Z,B)$ if in addition 
$$
a(D,X,B)< a(D,Y,B_Y)
$$ 
for any prime divisor $D$ on $X$ which is exceptional/$Y$. 

$\rm (2)$ Let $(Y/Z,B_Y)$ be a log minimal model of $(X/Z,B)$ such that 
$K_Y+B_Y$ is semi-ample$/Z$ so that there is a contraction $f\colon Y\to S/Z$ 
and an ample$/Z$ $\R$-divisor $H$ on $S$ such that $K_Y+B_Y\sim_\R f^*H/Z$. 
We call $S$ the log canonical model of $(X/Z,B)$ which is unique up to isomorphism$/Z$.

$\rm (3)$ On the other hand, we say that $(Y/Z,B_Y)$ is a Mori fibre space of $(X/Z,B)$ if 
there is a $K_Y+B_Y$-negative extremal contraction $Y\to T/Z$ such that $\dim T<\dim Y$, and if
$$
a(D,X,B)\le a(D,Y,B_Y)
$$ 
for any prime divisor $D$ on birational models of $X$ with strict inequality for any prime divisor $D$ on $X$ which is exceptional/$Y$.
\end{defn}

Note that in [\ref{B-mmodels}], it is not assumed that $\phi^{-1}$ does not contract divisors. However, since in this paper we are mainly concerned with constructing models for klt pairs, in that case our definition here is equivalent to that of [\ref{B-mmodels}].

\begin{lem}\label{l-ray}
Let $(X/Z,B+C)$ be a $\Q$-factorial lc pair where $B,C\ge 0$, 
$K_X+B+C$ is nef/$Z$, and $(X/Z,B)$ is dlt. Then, either $K_X+B$ is also nef/$Z$ or there is an extremal ray $R/Z$ such
that $(K_X+B)\cdot R<0$, $(K_X+B+\lambda C)\cdot R=0$, and $K_X+B+\lambda C$ is nef$/Z$ where
$$
\lambda:=\inf \{t\ge 0~|~K_X+B+tC~~\mbox{is nef/$Z$}\}
$$
\end{lem}
\begin{proof}
This is proved in [\ref{B-mmodels}, Lemma 2.7] assuming that $(X/Z,B+C)$ is dlt. 
We extend it to the lc case.

Suppose that $K_X+B$ is not nef$/Z$ and let $\{R_i\}_{i\in I}$ be the set of $(K_X+B)$-negative extremal rays/$Z$
and $\Gamma_i$ an extremal curve of $R_i$ [\ref{ordered}, Definition 1]. Let $\mu:=\sup \{\mu_i\}$ where
$$
\mu_i:=\frac{-(K_X+B)\cdot \Gamma_i}{C\cdot \Gamma_i}
$$

Obviously, $\lambda=\mu$ and $\mu\in (0,1]$. It is enough to
prove that $\mu=\mu_l$ for some $l$. 
By [\ref{ordered}, Proposition 1], there are positive real numbers
$r_1,\dots, r_s$ and a positive integer $m$ (all independent of $i$) such that
$$
(K_X+B)\cdot \Gamma_i=\sum_{j=1}^s\frac{r_jn_{i,j}}{m}
$$
where $-2(\dim X)m\le n_{i,j}\in\Z$.
On the other hand, by [\ref{log-models}, First Main Theorem 6.2, Remark 6.4] we can write 
$$
K_X+B+C=\sum_{k=1}^{t} r_k'(K_X+\Delta_k)
$$ 
where $r_1',\cdots, r_{t}'$ are positive 
real numbers such that for any $k$ we have: $(X/Z,\Delta_k)$ is lc with $\Delta_k$ being rational, and $(K_X+\Delta_k)\cdot \Gamma_i\ge 0$ for any $i$. Therefore, 
there is a positive integer $m'$ (independent of $i$) such that 
$$
(K_X+B+C)\cdot \Gamma_i=\sum_{k=1}^{t}\frac{r_k'n_{i,k}'}{m'}
$$
where $0\le n_{i,k}'\in\Z$. 

The set $\{n_{i,j}\}_{i,j}$ is finite.
Moreover,
$$
\frac{1}{\mu_i}= \frac{C\cdot \Gamma_i}{-(K_X+B)\cdot \Gamma_i}=\frac{(K_X+B+C)\cdot \Gamma_i}{-(K_X+B)\cdot \Gamma_i}+1
=-\frac{m\sum_k r_k'n_{i,k}'}{m'\sum_j r_jn_{i,j}}+1
$$

Thus, $\inf \{\frac{1}{\mu_i}\}=\frac{1}{\mu_l}$ for some $l$ and so $\mu=\mu_l$.
\end{proof}

\begin{defn}[LMMP with scaling]\label{d-scaling}
Let $(X/Z,B+C)$ be a lc pair such that $K_X+B+C$ is nef/$Z$, $B\ge 0$, and $C\ge 0$ is $\R$-Cartier. 
Suppose that either $K_X+B$ is nef/$Z$ or there is an extremal ray $R/Z$ such
that $(K_X+B)\cdot R<0$, $(K_X+B+\lambda_1 C)\cdot R=0$, and $K_X+B+\lambda_1 C$ is nef$/Z$ where
$$
\lambda_1:=\inf \{t\ge 0~|~K_X+B+tC~~\mbox{is nef/$Z$}\}
$$
 When $(X/Z,B)$ is $\Q$-factorial dlt, the last sentence follows from 
Lemma \ref{l-ray}. If $R$ defines a Mori fibre structure, we stop. Otherwise assume that $R$ gives a divisorial 
contraction or a log flip $X\bir X'$. We can now consider $(X'/Z,B'+\lambda_1 C')$  where $B'+\lambda_1 C'$ is 
the birational transform 
of $B+\lambda_1 C$ and continue the argument. That is, suppose that either $K_{X'}+B'$ is nef/$Z$ or 
there is an extremal ray $R'/Z$ such
that $(K_{X'}+B')\cdot R'<0$, $(K_{X'}+B'+\lambda_2 C')\cdot R'=0$, and $K_{X'}+B'+\lambda_2 C'$ is nef$/Z$ where
$$
\lambda_2:=\inf \{t\ge 0~|~K_{X'}+B'+tC'~~\mbox{is nef/$Z$}\}
$$
 By continuing this process, we obtain a 
special kind of LMMP$/Z$ which is called the \emph{LMMP$/Z$ on $K_X+B$ with scaling of $C$}; note that it is not unique. 
This kind of LMMP was first used by Shokurov [\ref{log-flips}].
When we refer to \emph{termination with scaling} we mean termination of such an LMMP.

\emph{Special termination  with scaling} means termination near $\rddown{B}$ of any sequence of log flips$/Z$ with scaling 
of $C$, i.e. after finitely many steps, the locus of the extremal rays in the process do not intersect $\Supp \rddown{B}$.

When we have a lc pair $(X/Z,B)$, we can always find an ample$/Z$ $\R$-Cartier divisor $C\ge 0$ such that 
$K_X+B+C$ is lc and nef$/Z$,  so we can run the LMMP$/Z$ with scaling assuming that all the 
necessary ingredients exist, eg extremal rays, log flips. 
\end{defn}

%%%%%%%%%%%%%%%%%%%%%%%%%%%%%%%%%%%%%%%%%%%%%
\subsection*{Finiteness of models.\\\\} 

({\bf P}) Let $X\to Z$ be a projective morphism of normal quasi-projective varieties, $A\ge 0$ a $\Q$-divisor on $X$, and 
$V$ a rational (i.e. with a basis consisting of rational divisors) finite dimensional affine subspace of the space of $\R$-Weil divisors on $X$. Define 
$$
\mathcal{L}_{A}(V)=\{B \mid 0\le (B-A)\in V, ~ \mbox{and $(X/Z,B)$ is lc} \}
$$
By [\ref{log-flips}, 1.3.2], $\mathcal{L}_{A}(V)$ is a rational polytope (i.e. a polytope with rational vertices) inside the rational affine space $A+V$.

\begin{rem}\label{r-local}
With the setting as in {\rm({\bf P})} above assume that $A$ is big$/Z$. 
Let $B\in \mathcal{L}_{A}(V)$ such that $(X/Z,B)$ is klt. Let $A'\ge 0$ be an ample$/Z$ $\Q$-divisor. Then, there is a rational number $\epsilon>0$ and an $\R$-divisor $G\ge 0$ such that $A\sim_\R \epsilon A'+G/Z$ and $(X/Z,B-A+\epsilon A'+G)$ is klt. Moreover, there is a neighborhood of $B$ in $\mathcal{L}_{A}(V)$ such that for any $B'$ in that neighborhood $(X/Z,B'-A+\epsilon A'+G)$ is klt. The point is that we can change $A$ 
and get an ample part $\epsilon A'$ in the boundary. So, when we are concerned with a problem locally around $B$ we feel free to assume that $A$ is actually ample by replacing it with $\epsilon A'$.
\end{rem}

\begin{lem}\label{l-ray2}
With the setting as in {\rm({\bf P})} above assume that $A$ is big$/Z$, $(X/Z,B)$ is klt of dimension $d$, and $K_X+B$ is nef$/Z$ where $B\in \mathcal{L}_{A}(V)$. 
 Then, there is $\epsilon>0$ (depending on $X,Z,V,A,B$)
such that if $R$ is a $K_X+B'$-negative extremal ray$/Z$ for some $B'\in \mathcal{L}_{A}(V)$ with $||B-B'||<\epsilon$ 
then $(K_X+B)\cdot R=0$.   
\end{lem}
\begin{proof}
This is proved in [\ref{ordered}, Corollary 9] in a more general situtation. Since $B$ is big$/Z$ and $K_X+B$ is nef$/Z$, the base 
point free theorem implies that $K_X+B$ is semi-ample$/Z$ hence there is a contraction $f\colon X\to S/Z$ and an ample$/Z$ 
$\R$-divisor $H$ on $S$ such that $K_X+B\sim_\R f^*H/Z$. We can write $H\sim_\R \sum a_iH_i/Z$ where $a_i>0$ 
and the $H_i$ are ample$/Z$ Cartier divisors on $S$. Therefore, there is $\delta>0$ such that for any curve 
$C/Z$ in $X$ either $(K_X+B)\cdot C=0$ or $(K_X+B)\cdot C>\delta$. 

Now let $\mathcal{C} \subset \mathcal{L}_{A}(V)$ be a rational polytope of maximal dimension which contains an open 
neighborhood of $B$ in $\mathcal{L}_{A}(V)$ and such that $(X/Z,B')$ is klt for any $B'\in \mathcal{C}$. Pick
$B'\in \mathcal{C}$ and let $B''$ be the point on the boundary of $\mathcal{C}$ such that $B'$ belongs to the line segment 
determined by $B,B''$. Let $R$ be a $K_X+B'$-negative extremal ray $R/Z$. If $(K_X+B)\cdot R>0$, then 
$(K_X+B'')\cdot R<0$ and there is a rational curve $\Gamma$ in $R$ such that 
$(K_X+B'')\cdot \Gamma\ge -2d$. Since $(K_X+B')\cdot \Gamma<0$, $(B'-B)\cdot \Gamma <-\delta$. 
If $||B-B'||$ is too small we cannot have 
$$
(K_X+B'')\cdot \Gamma=(K_X+B')\cdot \Gamma+(B''-B')\cdot \Gamma\ge -2d
$$
because $(B''-B')\cdot \Gamma$ would be too negative.
\end{proof}

\begin{thm}\label{t-finiteness}
Assume $\rm (1)$ of Theorem \ref{main} in dimension $d$. With the setting as in {\rm({\bf P})} above assume that $A$ is big$/Z$. Let $\mathcal{C}\subseteq \mathcal{L}_{A}(V)$ 
be a rational polytope such that $(X/Z,B)$ is klt for any $B\in \mathcal{C}$ where $\dim X=d$. Then, there are finitely many birational 
maps $\phi_i\colon X\bir Y_i/Z$ such that for any $B\in \mathcal{C}$ with $K_X+B$ pseudo-effective$/Z$, there 
is $i$ such that $(Y_i/Z,B_{Y_i})$ is a log minimal model of $(X/Z,B)$.
\end{thm}
\begin{proof}
Remember that as usual $B_{Y_i}$ is the birational transform of $B$. We may proceed locally, so fix $B\in \mathcal{C}$. If $K_X+B$ is not pseudo-effective$/Z$ then the same 
holds in a neighborhood of $B$ inside $\mathcal{C}$, so 
we may assume that $K_X+B$ is pseudo-effective$/Z$. By assumptions, $(X/Z,B)$ 
has a log minimal model $(Y/Z,B_Y)$. Moreover, the polytope $\mathcal{C}$ determines a rational polytope $\mathcal{C}_Y$ of 
$\R$-divisors on $Y$ by taking birational transforms of elements of $\mathcal{C}$. If we shrink $\mathcal{C}$ 
around $B$ we can assume that the inequality in (1) of Definition \ref{d-mmodel} is satisfied for every $B'\in \mathcal{C}$, that is, 
$$
a(D,X,B')< a(D,Y,B_Y')
$$ 
for any prime divisor $D\subset X$ contracted$/Y$. Moreover, a log minimal model of 
 $(Y/Z,B_Y')$ would also be a log minimal model of $(X/Z,B')$, for any $B'\in \mathcal{C}$. 
Therefore, we can replace $(X/Z,B)$ by $(Y/Z,B_Y)$ and assume from now on that $(X/Z,B)$ 
is a log minimal model of itself, in particular, $K_X+B$ is nef$/Z$.

Since $B$ is big$/Z$, by the base point 
free theorem, $K_X+B$ is semi-ample$/Z$ so there is a contraction $f\colon X\to S/Z$ such that $K_X+B\sim_\R f^*H/Z$ 
for some ample$/Z$ $\R$-divisor $H$ on $S$. Now by induction on the dimension 
of $\mathcal{C}$, we may assume that 
the theorem already holds over $S$ for all the points on the proper faces of $\mathcal{C}$, that is, 
there are finitely many birational maps $\psi_j\colon X\bir Y_j/S$ such that 
for any $B''$ on the boundary of $\mathcal{C}$ with $K_X+B''$ pseudo-effective$/S$, there is $j$ such that 
$(Y_j/S,B_{Y_j}'')$ is a log minimal model of $(X/S,B'')$.

By Lemma \ref{l-ray2}, if we further shrink $\mathcal{C}$ around $B$, then for any $B'\in \mathcal{C}$, any 
$j$, and any $K_{Y_j}+B_{Y_j}'$-negative extremal ray $R/Z$ we have the equality $(K_{Y_j}+B_{Y_j})\cdot R=0$. 
Note that all the pairs $(Y_j/Z,B_{Y_j})$ are klt and $K_{Y_j}+B_{Y_j}\equiv 0/S$ and nef$/Z$ because $K_X+B\equiv 0/S$.

Assume that  $B\neq B'\in \mathcal{C}$ such that $K_X+B'$ is pseudo-effective$/Z$, and let $B''$ be the unique point on 
the boundary of $\mathcal{C}$ such that $B'$ belongs to the line segment given by $B$ and $B''$. Since 
$K_X+B\equiv 0/S$, $K_X+B''$ is pseudo-effective$/S$, and $(Y_j/S,B_{Y_j}'')$ is a log minimal model of $(X/S,B'')$ for some $j$. 
So, $(Y_j/S,B_{Y_j}')$ is a log minimal model of $(X/S,B')$. Furthermore, $(Y_j/Z,B_{Y_j}')$ is a log minimal model of $(X/Z,B')$ because any $K_{Y_j}+B_{Y_j}'$-negative extremal ray $R/Z$ would be over $S$ by the last paragraph.
\end{proof}

%%%%%%%%%%%%%%%%%%%%%%%%%%%%%%%%%%%%%%%%%%%%%%%%%%%
\subsection*{Termination with scaling}

\begin{thm}\label{t-term}
Assume $\rm (1)$ of Theorem \ref{main} in dimension $d$ and let $(X/Z,B+C)$ be a klt pair of dimension $d$ 
where $B\ge 0$ is big$/Z$,  
$C\ge 0$ is $\R$-Cartier, and $K_X+B+C$ is nef$/Z$. Then, we can run the LMMP$/Z$ on $K_X+B$ with scaling of $C$ 
and it terminates.  
\end{thm}
\begin{proof}
Note that existence of klt log flips in dimension $d$ follows from the assumptions (see the proof of 
Corollary \ref{c-flips}). Run the LMMP$/Z$ on $K_X+B$ with scaling of $C$ and assume that we get an infinite sequence 
$X_i\bir X_{i+1}/Z_i$ of log flips$/Z$. We may assume that $X=X_1$. 
Let $\lambda_i$ be as in Definition \ref{d-scaling} and put $\lambda=\lim \lambda_i$. 
So, by definition, $K_{X_i}+B_i+\lambda_iC_i$ is nef$/Z$ and numerically zero over $Z_i$ 
where $B_i$ and $C_i$ are the birational transforms of $B$ and $C$ respectively. By taking a 
$\Q$-factorialisation of $X$, which exists by induction on $d$ and Lemma \ref{l-extraction-klt}, we can assume 
that all the $X_i$ are $\Q$-factorial. 

Let $H_1, \cdots, H_m$ be general ample$/Z$ Cartier divisors 
on $X$ which generate the space $N^1(X/Z)$. Since $B$ is big$/Z$, we may assume that 
$B-\epsilon (H_1+\cdots+H_m)\ge 0$ for some rational number $\epsilon>0$ (see Remark \ref{r-local}). Put $A=\frac{\epsilon}{2} (H_1+\cdots+H_m)$.
Let $V$  be the space generated by  
the components of $B+C$, and let $\mathcal{C}\subset \mathcal{L}_A(V)$ be a rational polytope of maximal dimension containing 
neighborhoods of $B$ and $B+C$ such that $(X/Z,B')$ is klt for any $B'\in \mathcal{C}$.
Moreover, we can choose $\mathcal{C}$ such that for each $i$ there is an ample$/Z$ $\Q$-divisor $G_i=\sum g_{i,j}H_{i,j}$ on $X_i$ 
with sufficiently small coefficients, 
where $H_{i,j}$ on $X_i$ is the birational transform of $H_j$, such that $(X_i,B_i+G_i+\lambda_iC_i)$ 
is klt and the birational transform of $B_i+G_i+\lambda_iC_i$ on $X$ belongs to $\mathcal{C}$.

Let $\phi_{i,j}\colon X_i\bir X_j$ be the birational map induced by the above sequence of log flips. Since $K_{X_i}+B_i+G_i+\lambda_iC_i$ is ample$/Z$ and since the log canonical model is unique, by Theorem \ref{t-finiteness}, there exist an infinite set $J\subseteq \N$ and a birational map $\phi\colon X=X_1\bir Y/Z$ such that $\psi_j:=\phi_{1,j}\phi^{-1}$ is 
an isomorphism for any $j\in J$. This in turn implies that $\phi_{i,j}$ is an isomorphism for any $i,j\in J$. This is not possible as any log flip increases some log discrepancies. 
\end{proof}

\begin{thm}\label{t-s-term}
Assume $\rm (1)$ of Theorem \ref{main} in dimension $d-1$ and let $(X/Z,B+C)$ be a $\Q$-factorial dlt pair of dimension 
$d$ where $B-A\ge 0$ for some ample$/Z$ $\R$-divisor $A\ge 0$, and $C\ge 0$. Assume that 
$(Y/Z,B_Y+C_Y)$ is a log minimal model of $(X/Z,B+C)$. Then, the special termination holds for the LMMP$/Z$ on $K_Y+B_Y$ with scaling of $C_Y$.
\end{thm}
\begin{proof}
Note that we are assuming that we are able to run a specific LMMP$/Z$ on $K_Y+B_Y$ with scaling of $C_Y$ otherwise there is nothing to prove, i.e. here we do not prove that such an LMMP$/Z$ exists but assume its existence. Suppose that 
we get a sequence $Y_i\bir Y_{i+1}/Z_i$ of log flips$/Z$ for such an LMMP$/Z$. Let $S$ be 
a component of $\rddown{B}$ and let $S_Y$ and $S_{Y_i}$ be its birational transform on $Y$ and $Y_i$ respectively. 

First suppose that we always have  $\lambda_j=1$
 in every step where $\lambda_j$ is as in Definition \ref{d-scaling}. Since $B-A\ge 0$ and since $A$ is ample$/Z$,  we can write 
$B+C\sim_\R A'+B'+C'/Z$ such that
 $A'\ge 0$ is ample/$Z$, $B',C'\ge 0$, $\rddown{B'}=\rddown{A'+B'+C'}=S$, 
$(X/Z,A'+B'+C')$ and $(Y/Z,A'_Y+B_Y'+C_Y')$ are plt, 
 and the LMMP/$Z$ on $K_Y+B_Y$ with scaling of $C_Y$ induces an LMMP/$Z$ on $K_Y+A_Y'+B_Y'$ 
with scaling of $C_Y'$. If $S_{Y_1}\neq 0$, by restricting to $S_Y$ and using a standard argument (cf. proof of [\ref{B-mmodels}, Lemma 2.11]) 
together with Theorem \ref{t-term}, we deduce that 
the log flips in the sequence  $Y_i\bir Y_{i+1}/Z_i$ do not intersect $S_{Y_i}$ for $i\gg 0$.

Now assume that we have  $\lambda_j<1$ for some $j$. Then, $\rddown{B+\lambda_j C}=\rddown{B}$ for any $j\gg 0$. So, we may assume that $\rddown{B+C}=\rddown{B}$. Since $B-A\ge 0$ and since $A$ is ample/$Z$, we can write 
$B\sim_\R A'+B'/Z$ such that $A'\ge 0$ is ample/$Z$, $B'\ge 0$, 
$\rddown{B'}=\rddown{A'+B'+C}=S$,  and 
 $(X/Z,A'+B'+C)$ and $(Y/Z,A_Y'+B_Y'+C_Y)$ are plt. The rest goes as before by restricting to $S_Y$.
\end{proof}
\bigskip
%%%%%%%%%%%%%%%%%%%%%%%%%%%%%%%%%%%%
\subsection*{Pl flips.\\\\}

We need an important result of Hacon-McKernan [\ref{HM}] which in turn is based on important works of 
Shokurov  [\ref{pl-flips}][\ref{log-flips}], Siu [\ref{Siu-invariance}] and Kawamata [\ref{Kaw}].

\begin{thm}\label{t-pl-flips}
Assume $\rm (1)$ of Theorem \ref{main} in dimension $d-1$. Then, pl flips exist in dimension $d$.
\end{thm}
\begin{proof}
By Theorem \ref{t-term} and Corollary \ref{c-flips} in dimension $d-1$, [\ref{HM}, Assumption 5.2.3] is satisfied in dimension $d-1$. Note that Corollary \ref{c-flips} in dimension $d-1$ easily follows from (1) of Theorem \ref{main} in dimension $d-1$ (see the proof of Corollary \ref{c-flips}). Now [\ref{HM}, Theorem 5.4.25, proof of Lemma 5.4.26] implies the result.\\
\end{proof}

%%%%%%%%%%%%%%%%%%%%%%%%%%%%%%%%%%%
\subsection*{Log minimal models}

\begin{lem}\label{l-extraction-klt}
Assume $\rm (1)$ of Theorem \ref{main} in dimension $d-1$. 
Let $(X/Z,B)$ be a klt pair of dimension
 $d$ and let $\{D_i\}_{i\in I}$ be a finite set of exceptional$/X$ prime divisors (on birational
 models of $X$) such that the log discrepancy $a(D_i,X,B)\le 1$. Then, there is a $\Q$-factorial klt pair $(Y/X,B_Y)$
 such that\\\\
 $\rm (1)$ $Y\to X$ is birational and $K_Y+B_Y$ is the crepant pullback of $K_X+B$,\\
 $\rm (2)$ the set of exceptional/$X$ prime divisors of $Y$ is exactly $\{D_i\}_{i\in I}$.\\
\end{lem}
\begin{proof}
 Let $f\colon W\to X$ be a log resolution of
$(X/Z,B)$ and let $\{E_j\}_{j\in J}$
be the set of prime exceptional divisors of $f$. We can assume that
for some $J'\subseteq J$, $\{E_j\}_{j\in J'}=\{D_i\}_{i\in I}$. Since $f$ is birational, there is an ample$/X$ $\Q$-divisor $H\ge 0$ on $W$ whose 
support is irreducible smooth and distinct from the birational transform of the components of $B$, and an $\R$-divisor $G\ge 0$ such that $H+G\sim_\R f^*(K_X+B)/X$. 
Moreover, there is $\epsilon>0$ such that $(X/Z,B+\epsilon f_*H+\epsilon f_*G)$ is klt. Now define
$$
K_W+\overline{B}_W:=f^*(K_{X}+B+\epsilon f_*H+\epsilon f_*G)+\sum_{j\notin J'} a(E_j,X,B+\epsilon f_*H+\epsilon f_*G)E_j
$$
 for which obviously there is an exceptional$/X$ $\R$-divisor $\overline{M}_W\ge 0$ such that 
$K_W+\overline{B}_W\sim_\R \overline{M}_W/X$ 
and $\theta(W/X,\overline{B}_W,\overline{M}_W)=0$.
By running the LMMP/$X$ on $K_W+\overline{B}_W$ with scaling of some ample$/X$ $\R$-divisor, and 
using the special termination of Theorem \ref{t-s-term} we get a log minimal model  of $(W/X,\overline{B}_W)$ which we 
may denote by $(Y/X,\overline{B}_Y)$. Note that here we only need pl flips to run the LMMP/$X$ because any extremal ray in the process 
 intersects some component of $\rddown{\overline{B}_W}$ negatively.

The exceptional divisor $E_j$ is contracted$/Y$ exactly when $j\notin J'$. By taking 
$K_Y+B_Y$ to be the crepant pullback of  $K_X+B$ we get the result.
\end{proof}

\begin{proof}(of Theorem \ref{t-mmodels})
We closely follow the proof of [\ref{B-mmodels}, Theorem 1.3]. Remember that the assumptions imply that pl flips exist in dimension 
$d$ by Theorem \ref{t-pl-flips} and that the special termination holds as in Theorem \ref{t-s-term}.
 
\emph{Step 1.} Since $B$ is big/$Z$, we can assume that it has a general ample$/Z$ component 
which is not a component of $M$  (see Remark \ref{r-local}).  
 By taking a log resolution we can further 
assume that the triple $(X/Z,B,M)$ is log smooth. To construct log minimal models in this situation we need 
to pass to a more general setting.

Let $\mathfrak{W}$ be the set of triples $(X/Z,B,M)$ which satisfy\\\\
(1) $(X/Z,B)$ is dlt of dimension $d$ and $(X/Z,B,M)$ is log smooth,\\
(2) $(X/Z,B)$ does not have a log minimal model,\\
(3) $B$ has a component which is ample$/Z$ but it is not a component of $\rddown{B}$ nor a component 
of $M$.\\

Obviously, it is enough to prove that $\mathfrak W$ is empty. Assume otherwise and
choose  $(X/Z,B,M)\in\mathfrak W$ with minimal $\theta(X/Z,B,M)$.

If $\theta(X/Z,B,M)=0$, then either $M=0$ in which case we already have a log minimal model,
or by running the LMMP/$Z$ on $K_X+B$ with scaling of a suitable ample$/Z$ $\R$-divisor we get a log 
minimal model because by the special termination of Theorem \ref{t-s-term},
flips and divisorial contractions will not intersect $\Supp \rddown{B}\supseteq \Supp M$ after finitely many steps.
This is a contradiction. 
Note that we need only pl flips here which exist by Theorem \ref{t-pl-flips}. We may then assume that $\theta(X/Z,B,M)> 0$.\\

\emph{Step 2.} Notation: for an $\R$-divisor $D=\sum d_iD_i$ we define $D^{\le 1}:=\sum d_i'D_i$ in which  $d_i'=\min\{d_i,1\}$. Now put
$$
\alpha:=\min\{t>0~|~~\rddown{(B+tM)^{\le 1}}\neq \rddown{B}~\}
$$

In particular, $(B+\alpha M)^{\le 1}=B+C$ for some $C\ge 0$ supported in $\Supp M$, and $\alpha M=C+M'$ 
where $M'$ is supported in $\Supp \rddown{B}$. Thus, outside $\Supp \rddown{B}$ 
we have $C=\alpha M$. The pair $(X/Z,B+C)$ is  $\Q$-factorial dlt and
$(X/Z,B+C,M+C)$ is a triple which satisfies (1) and (3) above. By construction
$$
\theta(X/Z,B+C,M+C)<\theta(X/Z,B,M)
$$
 so $(X/Z,B+C,M+C)\notin \mathfrak W$. Therefore, $(X/Z,B+C)$ has a log minimal model,
say $(Y/Z,B_Y+C_Y)$. By definition, $K_Y+B_Y+C_Y$ is nef/$Z$.\\

\emph{Step 3.} Now run the LMMP$/Z$ on $K_Y+B_Y$ with scaling of $C_Y$. Note that we only need pl flips here 
because every extremal ray contracted in the process would have negative intersection with some component of $\rddown{B}$ by the properties of $C$ mentioned in Step 2. 
By the special termination of Theorem \ref{t-s-term},
after finitely many steps, $\Supp \rddown{B}$ does not intersect the extremal rays
contracted by the LMMP hence we end up with a model $Y'$ on which $K_{Y'}+B_{Y'}$ is nef/$Z$. Clearly, $(Y'/Z,B_{Y'})$
is a nef model of  $(X/Z,B)$ but may not be a log minimal model because the inequality in  (1) of Definition \ref{d-mmodel} may not be satisfied.\\

\emph{Step 4.} Let
$$
\mathcal T=\{t\in [0,1]~|~ (X/Z,B+tC)~~\mbox{has a log minimal model}\}
$$
Since $1\in\mathcal T$, $\mathcal T\neq \emptyset$. Let $t\in\mathcal T\cap (0,1]$ and let 
$(Y_t/Z,B_{Y_t}+tC_{Y_t})$ be any log minimal model of $(X/Z,B+tC)$. Running the LMMP/$Z$ on $K_{Y_t}+B_{Y_t}$ with scaling
of $tC_{Y_t}$ shows that there is $t'\in(0,t)$ such that $[t',t]\subset \mathcal{T}$
because the inequality required in (1) of Definition \ref{d-mmodel} is an open condition. The
LMMP terminates for the same reasons as in Step 3 and we note again that the log flips 
required are all pl flips.\\

\emph{Step 5.} Let $\tau=\inf \mathcal T$. If $\tau\in\mathcal{T}$, then by Step 4, $\tau=0$
and so we are done by deriving a contradiction. Thus, we may assume that $\tau\notin\mathcal{T}$. In this case, there is a sequence
$t_1>t_2>\cdots$ in $ \mathcal{T}\cap (\tau,1]$
such that $\lim_{k\to +\infty} t_k=\tau$. For each $t_k$ let $(Y_{t_k}/Z,B_{Y_{t_k}}+t_kC_{Y_{t_k}})$ be any log minimal model of
$(X/Z,B+t_kC)$ which exists by the definition of $\mathcal{T}$ and from which we get a nef model $(Y_{t_k}'/Z,B_{Y_{t_k}'}+\tau C_{Y_{t_k}'})$
for $(X/Z,B+\tau C)$ by running the LMMP/$Z$ on $K_{Y_{t_k}}+B_{Y_{t_k}}$ with scaling of $t_kC_{Y_{t_k}}$.
Let $D\subset X$ be a prime divisor contracted/$Y_{t_k}'$. If $D$ is contracted/$Y_{t_k}$, then 
$$
a(D,X,B+t_k C) < a(D,Y_{t_k},B_{Y_{t_k}}+t_kC_{Y_{t_k}})
$$
$$
\le a(D,Y_{t_k},B_{Y_{t_k}}+\tau C_{Y_{t_k}})\le a(D,Y_{t_k}',B_{Y_{t_k}'}+\tau C_{Y_{t_k}'})
$$
 but if $D$ is not contracted/$Y_{t_k}$ we have
$$
a(D,X,B+t_k C) =a(D,Y_{t_k},B_{Y_{t_k}}+t_kC_{Y_{t_k}})
$$
$$
\le a(D,Y_{t_k},B_{Y_{t_k}}+\tau C_{Y_{t_k}}) <a(D,Y_{t_k}',B_{Y_{t_k}'}+\tau C_{Y_{t_k}'})
$$
because $(Y_{t_k}/Z,B_{Y_{t_k}}+t_kC_{Y_{t_k}})$ is a log minimal model of $(X/Z,B+t_kC)$ and  $(Y_{t_k}'/Z,B_{Y_{t_k}'}+\tau C_{Y_{t_k}'})$
is a log minimal model of $(Y_{t_k}/Z,B_{Y_{t_k}}+\tau C_{Y_{t_k}})$. Thus, in any case we have 
$$
a(D,X,B+t_k C) < a(D,Y_{t_k}',B_{Y_{t_k}'}+\tau C_{Y_{t_k}'})
$$

Replacing the sequence $\{t_k\}_{k\in \N}$ with a subsequence, we can assume that all the induced rational maps $X\bir Y_{t_k}'$
contract the same components of $B+\tau C$. Now an easy application of the negativity lemma implies (cf. [\ref{B-mmodels}, Claim 2.10]) that the log discrepancy $a(D,Y_{t_k}',B_{Y_{t_k}'}+\tau C_{Y_{t_k}'})$ is independent of $k$.
Therefore, each $(Y_{t_k}',B_{Y_{t_k}'}+\tau C_{Y_{t_k}'})$ 
 is a nef model of $(X/Z,B+\tau C)$ such that
$$
 a(D,X,B+\tau C)=\lim_{l\to +\infty} a(D,X,B+t_l C)\leq a(D,Y_{t_k}',B_{Y_{t_k}'}+\tau C_{Y_{t_k}'})
$$
for any prime divisor $D\subset X$ contracted/$Y_{t_k}'$.\\

\emph{Step 6.} To get a log minimal model of  $(X/Z,B+\tau C)$ we just need to extract those prime divisors $D$ on $X$ 
contracted$/Y_{t_{k}}'$ for which 
$$
 a(D,X,B+\tau C)=a(D,Y_{t_k}',B_{Y_{t_k}'}+\tau C_{Y_{t_k}'})
$$
Since $B$ has a component which is ample$/Z$, we can find $\Delta$ on $X$ such that $\Delta\sim_{\R} B+\tau C/Z$ and
such that $(X/Z,\Delta)$  and $(Y_{t_k}'/Z,\Delta_{Y_{t_k}'})$ are klt (see Remark \ref{r-local}). Now we can apply Lemma \ref{l-extraction-klt}
to construct a crepant model of $(Y_{t_k}'/Z,\Delta_{Y_{t_k}'})$ which would be a log minimal model of $(X/Z,\Delta)$. This in turn induces a 
log minimal model of $(X/Z,B+\tau C)$.  Thus, $\tau\in\mathcal T$ and this gives a contradiction. Therefore, $\mathfrak{W}=\emptyset$.
\end{proof}

%%%%%%%%%%%%%%%%%%%%%%%%%%%%%%%%%%%%%%%%%%%%%%%%%
%%%%%%%%%%%%%%%%%%%%%%%%%%%%%%%%%%%%%%%%%%%%%%%%%

\section{Nonvanishing}

In this section we are going to prove the theorem 1.5, which is a numerical version of the corresponding result obtained in 
[\ref{BCHM}] (see equally [\ref{kawa3}], [\ref{Dru}] for interesting presentations of [\ref{BCHM}]). 

\bigskip

\subsection*{Preliminaries.\\\\}

\medskip

\noindent During the following subsections, we will give a complete proof of the next
particular case of the theorem 1.5 (the {\sl absolute case} $Z= \{z\}$).
\begin{thm}
Let $X$ be a smooth projective manifold, and let 
$B$ be an $\R$-divisor such that :

\begin{enumerate}
\item The pair $(X,B)$ is klt, and $B$ is big ;
\smallskip

\item The adjoint bundle $K_X+ B$ is pseudo-effective.
\end{enumerate}
\noindent Then there exist an effective 
$\R$-divisor
$\sum_{j= 1}^N\nu^j[Y_j]$ numerically equivalent with $K_X+ B$. \hfill\qed
\end{thm}

\medskip
\noindent 
We recall that by definition a big divisor $B$ contains in its cohomology class 
a current 
$$\Theta_B:= \omega_B+ [E]\leqno (1)$$
where $\omega_B$ is a K\"ahler metric, and $[E]$ is the current of integration associated to
an effective $\R$-divisor $E$. This is just a reformulation of the usual Kodaira lemma, except that in algebraic geometry one usually denotes the decomposition (1) by $B= H+ E$, where $H$ is ample ; the $\omega_B$ above is a smooth, positive representative of $c_1(H)$. 

Moreover, the pair $(X, B)$ is klt and $X$ is assumed to be non-singular, thus we have 
$$B= \sum b^jZ_j\leqno (2)$$
where $b^j$ are positive reals, $(Z_j)$ is a finite set of hypersurfaces of $X$ such that
$$\prod_j|f_j|^{-2b^j}\in L^1(\Omega)\leqno (3)$$
for each coordinate set $\Omega\subset X$, where $Z_j\cap \Omega= (f_j= 0)$.\

Therefore, by considering a convex combination of the objects in (1) and (2), we can assume from the beginning that the $\R$-divisor $E$ satisfy the integrability condition (3) : this can be seen as a
metric counterpart of the hypothesis (1) in the statement 3.1.

\medskip

Let $L$ be a pseudo-effective $\R$-divisor on $X$ ; we denote its {\sl numerical dimension} by $\num(L)$. The formal definition will not be reproduced here (the interested reader can 
profitably consult the references [\ref{naka}], [\ref{Bou}]), however, in order to gain some intuition about it, let us mention that if $L$ has a Zariski decomposition, then $\num(L)$ is the familiar numerical dimension of the nef part. 

\smallskip 

\noindent The statements which will follow assert the existence of geometric objects in the Chern class of $L$ and its approximations, according to the size of its numerical dimension.
The first one is due to N. Nakayama in [\ref{naka}] (see also the transcendental generalization by S. Boucksom, [\ref{Bou}]).

\begin{thm} {\rm ([\ref{Bou}], [\ref{naka}])}Let $L$ be a pseudo-effective $\R$-divisor 
such that $\num(L)= 0$. Then there exist an effective $\R$-divisor
$$\Theta~:= \sum_{j= 1}^\rho\nu^{j}[Y_j]\in \{\alpha\}.$$

\end{thm}
\noindent For a more complete discussion about the properties of the divisor 
$\Theta$ above we refer to the article [\ref{naka}].\hfill\qed

\smallskip
\noindent Concerning the pseudoeffective classes in $\NS_R(X)$ whose numerical dimension is strictly greater than 0, we have the following well-known statement.

\begin{thm} {\rm ([\ref{rob}])} Let $X$ be a projective manifold, let $L$ be a pseudo-effective 
$\R$-divisor, such that 
$\num (L)\geq 1$. Let $B$ be a big $\R$-divisor. 
Then for any $x\in X$ and $m\in \Z_+$ there exist 
an integer $k_m$ and a representative
$$T_{m, x}:= [D_m]+ \omega_m\equiv mL+B$$
where $D_m$ is an effective $\Q$-divisor and $\omega_m$ is a K\"ahler metric 
such that $\nu (D_m, x)\geq k_m$ and 
$k_m\to\infty$ as $m\to\infty$. \\ \hfill\qed
\end{thm}

%\bigskip
%%%%%%%%%%%%%%%%%%%%%%%%%%%%%%%%%%%%%%%%%%%%%%%%%%%%%%%%%%%%%%%%%%%%%%%%%%%%%%%%
\subsection*{Dichotomy.\\\\}

%\medskip 

\noindent We start now the actual proof of 3.1 and denote by $\nu$ the numerical dimension of the 
divisor ${K_X}+ B$. We proceed as in [\ref{shok}], [\ref{kawa2}], [\ref{BCHM}], [\ref{siu2}].

\smallskip
\noindent $\bullet$ If $\nu= 0$, then the theorem 3.5 is a immediate consequence of 3.6, so this first case is completely settled.\hfill\qed

\smallskip
\noindent $\bullet$ The second case $\nu\geq 1$ is much more involved~~; we are going to use induction on the dimension of the manifold. Up to a certain point, our arguments are very similar to the classical approach of Shokurov (see [\ref{shok}]) ~; perhaps the main difference is the 
use of the {\sl invariance of plurigenera} extension techniques as a substitute for the Kawamata-~Viehweg 
vanishing theorem in the classical case.
\medskip

Let $G$ be an ample bundle on $X$, endowed with a smooth metric whose curvature form is denoted by $\omega_G$; by hypothesis, the $\R$-divisor $K_X+ B$ is pseudo-effective, thus for each positive $\varepsilon$, there exists
an effective  $\R$--divisor
$$\Theta_{K_X+B, \varepsilon}\equiv K_X+ B+ \varepsilon G.$$
We denote by $W_\varepsilon$ the support of the divisor $\Theta_{K_X+B, \varepsilon}$ and we
consider a point $x_0\in X\setminus {\cup_\varepsilon  W_\varepsilon}$.
Then the statement 3.3 provides us with a current 
$$T= [D_m]+ \omega_m\equiv m(K_X+ B)+ B$$
such that $\nu(D_m, x_0)\geq 1+ \dim (X)$. The integer $m$ will be fixed during the rest of the proof.
\smallskip

\noindent The next step in the classical proof of Shokurov would be to
consider the log-canonical threshold of $T$, in order to use an inductive argument.
However, under the assumptions of 3.1 we cannot use exactly the same approach, since unlike in the nef context, the restriction of a pseudo-effective class to an arbitrary hypersurface may not be pseudo-effective. In order to avoid such an unpleasant surprise,
we introduce now our substitute for the log canonical threshold (see [\ref{mp2}] for an interpretation of the quantity below, and also [\ref{BCHM}] for similar considerations).

\medskip

Let $\mu_0~: \wt X\to X$ be a common log resolution of the singular part of $T$ and $\Theta_B$. By this we mean that $\mu_0$ is the composition of a sequence of 
blow-up maps with non-singular centers, such that we have

$$\mu_0^\star(\Theta_B)= \sum_{j\in J}a^{j}_B[Y_{j}]+ \wt \Lambda_B\leqno(4) $$

$$\mu_0^\star(T)= \sum_{j\in J}a^{j}_T[Y_{j}]+ \wt \Lambda_T
\leqno(5) $$

$$K_{\wt X/X}=  \sum_{j\in J}a^{j}_{\wt X/X}[Y_j]
\leqno(6) $$
where the divisors above are assumed to be non-singular and to have normal crossings, and $\wt \Lambda_B, \wt \Lambda_T$ are smooth (1,1)--forms.

\noindent Now the family of divisors $\displaystyle \Theta_{K_X+ B, \varepsilon}$ enter into the picture.
Let us consider its inverse image {\sl via} the map $\mu_0$ ~:

$$\mu_0^\star\big(\Theta_{K_X+ B, \varepsilon}\big)= \sum_{j\in J}a^{j}_{K_X+ B, \varepsilon}[Y_j]+ \wt \Lambda_{K_X+ B, \varepsilon}$$
where $\wt \Lambda_{K_X+ B, \varepsilon}$ in the relation above is an effective 
$\R$-divisor, whose support does not contain any of the hypersurfaces 
$(Y_j)_{j\in J}$. 

The set $J$ is finite and given independently of $\varepsilon$, so we can assume that the following limit exists
$$a^{j}_{K_X+ B}:= \lim _{\varepsilon\to 0}a^{j}_{K_X+ B, \varepsilon}.$$
For each $j\in J$, let $\alpha_j$ be non-singular representative of the 
Chern class of the bundle associated to $Y_j$ ; by the preceding equality we have
$$\mu_0^\star\big(\Theta_{K_X+ B, \varepsilon}\big)\equiv \sum_{j\in J}a^{j}_{K_X+ B}[Y_j]+ \wt \Lambda_{K_X+ B, \varepsilon}+ \sum_{j\in J}\delta^j_\varepsilon\alpha_j\leqno (7)$$
where $\delta^j_\varepsilon:= a^{j}_{K_X+ B, \varepsilon}-a^{j}_{K_X+ B}$.
We denote along the next lines by $\wt D$ the $\varepsilon$-free part 
of the current above. In conclusion, we have organized the previous terms 
such that $\mu_0$ appears as a {\sl partial log-resolution} for the family of divisors
$\displaystyle (\Theta_{K_X+ B, \varepsilon})_{\varepsilon > 0}$. 

\medskip

Given any real number $t$, consider the following quantity

$$
\mu_0^\star \big(K_X + t(T - \Theta_{B})+  \Theta_B\big) ~;$$
it is numerically equivalent to the current
$$K_{\wt X}+ (1+ mt)\wt D + (1-t)\wt \Lambda_B+ t\wt \Lambda_T
+  \sum_{j\in J}\gamma^j(t)[Y_{j}],
$$
where we use the following notations
$$\gamma^j(t)~:= ta^{j}_T+ (1-t)a^{j}_B- (1+ mt)a^{j}_{K_X+B}- 
a^j_{\wt X/X}.$$
We have $\mu^\star\big(\Theta_{K_X+ B}\big)\equiv \wt D+ 
\wt \Lambda_{K_X+ B, \varepsilon}+ \sum_{j\in J}\delta^j_\varepsilon\alpha_j$, 
and on the other hand the cohomology class of the current
$$t(T - \Theta_{B})+ \Theta_B- (1+ mt)\big(\Theta_{K_X+ B, \varepsilon}-\varepsilon \omega_G\big)$$
is equal to the first Chern class of $X$, so
by the previous relations we infer that the currents

$$ 
(1+ mt)\big(\wt \Lambda_{K_X+ B, \varepsilon}+ \sum_{j\in J}\delta^j_\varepsilon\alpha_j- \varepsilon\omega_G\big) \leqno (8)$$
and

$$\Theta_{\wh \omega}(K_{\wt X})+ 
\sum_{j\in J}\gamma^j(t)[Y_{j}]+ (1-t)\wt \Lambda_B+ t\wt \Lambda_T
\leqno (9)$$
are numerically equivalent, for any $t\in \R$. 

We use next the strict positivity of $B$, in order  to modify slightly the inverse image 
of $\Theta_B$ {\sl within the same cohomology class}, so that
we have ~:

\begin{enumerate}

\item[(i)] The real numbers 
$$\frac{1+a^{j}_{K_X+B}+ a^j_{\wh X/X}- a^{j}_B}{a^{j}_T- a^{j}_B- ma^{j}_{K_X+B}}$$
are distinct ~; \smallskip

\item[(ii)] The klt hypothesis in 3.1 is preserved, i.e. $a^{j}_B-a^{j}_{\wt X/X}< 1$ ~; \smallskip

\item[(iii)] The (1,1)--class $\{\Lambda_{B}\}$ contains a K\"ahler metric. 
\end{enumerate}
\smallskip

\noindent The arguments we use in order to obtain the above properties are quite standard ~: it is the so-called {\sl tie-break method}, therefore we will skip the details.
\medskip

Granted this, there exist a unique index say $0\in J$ and a positive real  
$\tau$ such that $\gamma^0(\tau)= 1$ and $\gamma^j(\tau)< 1$ for $j\in J\setminus \{0\}$. Moreover, we have $0< \tau< 1$, by the klt hypothesis and the concentration of the singularity of $T$ at the point $x_0$.

We equally have the next numerical identity 
$$(1+ m\tau)\big(\wt \Lambda_{K_X+ B, \varepsilon}+ \sum_{j\in J}\delta^j_\varepsilon\alpha_j- \varepsilon\omega_G\big)+ \wt H\equiv  K_{\wt X}+ \wt S+ \wt B \leqno(10)$$
where $\wt B$ is the $\R$-divisor
$$\wt B:= \sum_{j\in J_p} 
\gamma^j(\tau)
[Y_{j}] + (1-\tau)\wt \Lambda_B+ \tau\wt \Lambda_T \leqno(11) 
 $$
 and we also denote by
 $$\wt H~:= -\sum_{j\in J_n}\gamma^j(\tau)
  [Y_{j}].
\leqno(12)$$
The choice of the partition of $J= J_p\cup J_p\cup \{0\}$ is such that the coefficients 
of the divisor part in (11) are in $[0, 1[$, the $\R$ divisor $\wt H$ is effective, and of course the coefficient $\gamma^0(\tau)= 1$ corresponds to $\wt S$. 
\medskip

\noindent In order to apply induction, we collect here the main features of the objects constructed above.

\smallskip

\noindent {$\bullet$} In the first place, 
the $\R$-divisor $\sum_{j\in J_p} 
\gamma^j(\tau)
[Y_{j}]$ is klt, and the smooth $(1,1)$--form $(1-\tau)\wt \Lambda_B+ \tau\wt \Lambda_T$ is positive definite ; thus the $\R$-divisor in (?) is big and klt. Moreover, its restriction to $\wt S$ has the same properties.

\smallskip

\noindent {$\bullet$} There exist 
an effective $\R$--divisor $\Delta$ such that 
$$
\wt B+ \wt S+ \Delta \equiv \mu_0^\star\big(B+ \tau m(K_X+B)\big)
$$

\noindent Indeed, the expression of $\Delta$ is easily obtained as follows
$$\Delta~:= \sum_{j\in J_p\cup \{0\}}\big((1+ m\tau)a^j_{K_X+ B}+ a^j_{\wt X/X}\big)[Y_j]+ \sum_{j\in J_n}(\tau a^j_T+ (1-\tau)a^j_B)[Y_j].$$

\noindent Therefore, it is enough to produce an effective $\R$--divisor 
numerically equivalent to $K_{\wt X}+ \wt S+ \wt B$ in order to complete the proof of the theorem 3.1.
\smallskip

\noindent {$\bullet$} The adjoint bundle $K_{\wt X}+ \wt S+ \wt B$ and its
restriction to $\wt S$ are pseudo-effective by the relation (10).

\smallskip

\noindent {$\bullet$} By using a sequence of blow-up maps, we can even assume that the components $\displaystyle (Y_j)_{j\in J_p}$
in the decomposition (11) have empty mutual intersections. Indeed, this is a simple --but nevertheless crucial!-- classical result, which we recall next. 

We denote by 
$\Xi$ an effective $\R$-divisor, whose support do not contain $\wt S$, such that 
$\Supp \Xi\cup \wt S$ has normal crossings and such that its coefficients are strictly smaller than 1.

\begin{lem} {\rm (see e.g. [\ref{HM}])} There exist a birational map $\mu_1~: \wh X\to \wt X$ such that 
$$\mu_1^\star (K_{\wt X}+ \wt S+ \Xi)+ E_{\wh X}= K_{\wh X}+ S+ \Gamma$$
where $E_{\wh X}$ and $\Gamma$ are effective with no common components, 
$E_{\wh X}$ is 
exceptional and $S$ is the proper transform of $\wt S$~~; moreover, the support of the divisor $\Gamma$ has normal crossings, 
its coefficients are strictly smaller than 1 and the intersection of any two components is empty.
\hfill\qed
\end{lem}
We apply this result in our setting with $\Xi:= \sum_{j\in J_p} 
\gamma^j(\tau)
[Y_{j}]$, and we summarize the discussion in this paragraph in the next statement (in which we equally adjust the notations).

\begin{prop}There exist a birational map $\mu~: \wh X\to X$ and 
an $\R$-divisor  
$$\wh B\equiv \sum_{j\in J}\nu^{j}Y_j+ \wh\Lambda_B$$
on $\wh X$, where 
$0< \nu^{j}< 1$, the hypersurfaces $Y_j$ above are smooth, they have empty mutual intersection and moreover the following hold ~:

\begin{enumerate}

\item  There exist a family of closed $(1, 1)$--currents 
$$\Theta_\varepsilon := \Delta_{K_X+ B, \varepsilon}+ \alpha_\varepsilon$$
numerically equivalent with $K_{\wh X}+ S+ \wh B$
where $S\subset \wh X$ is a non-singular hypersurface which has transversal intersections with $(Y_j)$, and where $\Delta_{K_X+ B, \varepsilon}$
is an effective $\R$-divisor whose support is disjoint from the 
set $(S, Y_j)$, and finally $\alpha_\varepsilon$ is a non-singular (1,1)-form, greater than $-\varepsilon\omega$~; 
\smallskip

\item  There exist a map 
$\mu_1~:\wh X\to \wt X$ such that $S$ is not $\mu_1$--exceptional, and such that 
$\wh\Lambda_B$ is greater than the inverse image of a K\"ahler metric on $\wt X$
via $\mu_1$. Therefore, the form $\wh\Lambda_B$ is positive defined at the generic point of $\wh X$, and so is its restriction to the generic point of $S$~~;
\smallskip

\item  There exist an effective $\R$-divisor $\Delta$ on $\wh X$
such that 
$$\wh B+ S+ \Delta=  \mu^\star \big(B+ \tau m(K_X+B)\big)+ E$$
where $E$ is $\mu$--exceptional. \\ \hfill\qed
 
\end{enumerate}

\end{prop}

%%%%%%%%%%%%%%%%%%%%%%%%%%%%%%%%%%%%%%%%%

%\bigskip 

\subsection*{Restriction and induction.\\\\}

%\medskip
\noindent We consider next the restriction to $S$ of the currents $\Theta_\varepsilon$ above 
$$\Theta_{\varepsilon|S}\equiv K_{S}+ \wh B_{|S} ;\leqno(13) $$
we have the following decomposition
$$\Theta_{\varepsilon|S}= \sum_{j\in J}\rho^{\varepsilon, j}Y_{j|S}+ R_{\varepsilon}\leqno(14)$$
where the coefficients $(\rho^{\varepsilon, j})$ are positive real numbers, 
and $R_\varepsilon$ above is the closed current given by the restriction to
$S$ of the differential form
$$\alpha_\varepsilon\leqno (15)$$
plus the part of the restriction to $S$ of the $\R$-divisor 
$$\Delta_{K_X+ B, \varepsilon}\leqno (16)$$
which is disjoint from the family $(Y_{j|S})$. Even if the differential form in (15) may not be positive, nevertheless we can assume that we have 
$$R_\varepsilon\geq -\varepsilon \omega_{|S}$$
for any $\varepsilon> 0$. We remark that the coefficients $\rho^{\varepsilon, j}$
in (14) may be {\sl positives}, despite of the fact the $Y_j$ {\sl does not belongs} to the support of $\Theta_\varepsilon$, for any $j\in J$.
\smallskip

For each index $j\in J$ we will assume that the next limit
$$\rho^{\infty, j}~:= 
\lim_{\varepsilon\to 0}\rho^{\varepsilon, j}$$
exist, and we introduce the following notation
$$I~:= \{j\in J ~: \rho^{\infty, j}\geq \nu^{j}\}.\leqno(17) $$
The numerical identity in 3.5, (1) restricted to $S$ coupled with (14) show that we have
$$\sum_{j\in I}(\rho^{\infty, j}- \nu^{j})[Y_{j|S}]+
R_{\varepsilon}+ \sum_{j\in J}(\rho^{\varepsilon, j}-\rho^{\infty, j})[Y_{j|S}]
\equiv K_S+ B_S\leqno(18)$$
where $\displaystyle {B_S}$ is the current 
$$\sum_{j\in J\setminus I}(\nu^{j}-\rho^{\infty, j})[Y_{j|S}]+\wh \Lambda_{B|S}.$$
\medskip

\noindent We are now in good position to apply induction :
\smallskip

\noindent $\bullet$ The $\R$-divisor ${B_S}$ is big and klt on $S$. Indeed, this follows by (2) and the properties of $\wh B$ in 3.5,  and the 
definition of the set $I$, see (17).
\smallskip

\noindent $\bullet$ The adjoint divisor
$K_S+ {B_S}$ is pseudoeffective, by the relation (18).
\medskip

\noindent Therefore, we can apply the induction hypothesis ~: there exist 
 a non-zero, effective $\R$-divisor, which can be written as
$$T_S~:= \sum_{i\in K}\lambda^{i}[W_i]$$
(where $W_i\subset S$ are hypersurfaces) which is numerically equivalent to
$K_S+ B_S$. We consider now the current
$$\wh T_S~:= \sum_{i\in K}\lambda^{i}[W_i]+ 
\sum_{j\in J\setminus I}\rho^{\infty, j}[Y_{j|S}]+  
\sum_{j\in I}\nu^{j}[Y_{j|S}]~~;\leqno(19)$$
from the relation (18) we get 
$$\wh T_S\equiv K_{\wh X}+ S+ \wh B_{|S}.\leqno(20)$$

\medskip
\noindent It is precisely the $\R$-divisor $\wh T_S$ above who will ``generate" the section we seek, in the following manner. We first use a diophantine argument, in order to obtain a {\sl simultaneous approximation} of $T_S$ and $\wh B$ with
a $\Q$-divisor, respectively $\Q$-line bundle, such that the relation (20) above still holds. The next step is to use a trick by Shokurov (adapted to our setting)
and finally the main ingredient is an extension result for pluricanonical forms.
All this will be presented in full details in the next three subsections. \\
\hfill\qed

%%%%%%%%%%%%%%%%%%%%%%%%%%%%%%%%%%%%%%%%%

%\bigskip

\subsection*{Approximation.\\\\}

%\medskip
\noindent In this paragraph we recall the following diophantine approximation lemma (we refer to [\ref{mp2}] for a complete proof).

\begin{lem} For each $\eta> 0$, there exist a positive integer $q_\eta$, a $\Q$--line bundle 
$\wh B _\eta$ on $\wh X$ and a $\Q$-divisor
$$\wh T_{S,\eta}~:= \sum_{i\in K}\lambda^{i}_\eta[W_i]+ \sum_{j\in J\setminus I}
\rho^{\infty, j}_\eta[Y_{j|S}]+  
\sum_{j\in I}\nu^{j}_{\eta}[Y_{j|S}]\leqno(21)$$
on $S$ such that ~:

\begin{enumerate}

\item [A.1]The multiple $q_\eta\wh B _\eta$ is a genuine line bundle, and the numbers
$$(q_\eta\lambda^{i}_\eta)_{i\in K}, (q_\eta\nu^{j}_\eta)_{j\in J}, 
(q_\eta\rho^{\infty, j}_\eta)_{j\in J}$$ 
are integers~~;
\smallskip

\item [A.2] We have $\displaystyle \wh T_{S,\eta}\equiv K_{\wh X}+ S+ \wh B _{\eta |S}$ ~;
\smallskip

\item [A.3] We have $\Vert q_\eta\big(\wh B- \wh B _\eta\big)\Vert< \eta$,
$|q_\eta\big(\lambda^{i}_\eta- \lambda^{i}\big)| <\eta$ and the analog relation 
for the $(\rho^{\infty, j}, \nu^j)_{j\in J}$ 
%$\wh T_{S,\eta}$ 
(here $\Vert \cdot\Vert$ denotes any norm on 
the real Neron-Severi space of $\wh X$)~~;
\smallskip

\item [A.4]For each $\eta_0> 0$, there exist a finite family $(\eta_j)$ such that the class
$\{K_{\wh X}+ S+ \wh B\}$ belong to the convex hull of  
$\displaystyle \{K_{\wh X}+ S+ \wh B _{\eta_j}\}$
where $0< \eta_j< \eta_0$.\hfill\qed

\end{enumerate}

\end{lem}

\medskip

\noindent {\bf Remark.} Even if we do not reproduce here the arguments of the proof 
(again, see [\ref{mp2}]), we present an interpretation of it, due to S. Boucksom. Let 
$N:= |J|+ |K|$ ; we consider the map 
$$l_1: \R^N\to \NS_R(S)$$
defined as follows. To each vector $(x^1,..., x^N)$, it corresponds the class
of the $\R$-divisor $\sum_{i= 1}^ {|K|}x^iW_i+ \sum_{j= 1+|K|}^Nx^jY_{j|S}$.
We define another linear map
$$l_2 : \NS_R(X)\to \NS_R(S)$$
which is given by the restriction to $S$. We are interested in the set 
$${\mathcal I}:= (x^1,..., x^N; \tau)\in \R^N\times \NS_R(X)$$
such that 
$l_1(x)= l_2(\tau)$; it is a vector space, which is moreover defined over $\Q$
(since this is the case for both maps $l_1$ and $l_2$). Now our initial data 
$(T_S, \{K_X+ S\}+ \theta_{\wh L})$ corresponds to a point of the above fibered product, and the claim of the lemma is that given a point in a vector subspace defined over $\Q$, we can approximate it with rational points satisfying the Dirichlet condition.

%%%%%%%%%%%%%%%%%%%%%%%%%%%%%%%%%%%%%%%%%

%\bigskip
\subsection*{A trick by V. Shokurov.\\\\}

%\medskip 
\noindent Our concern in this paragraph will be to ``convert" the effective 
$\Q$-divisor $\wh T_{S,\eta}$ into
a genuine section $s_\eta$ of the bundle 
$\displaystyle q_\eta\big(K_{\wh X}+ S+ \wh B _\eta\big)$. To this end, we will apply a
classical argument of Shokurov, in the version revisited by Siu in his recent work 
[\ref{siu2}].
A crucial point is that by a careful choice of the metrics we use, the $L^2$ estimates will allow us to have a very precise information concerning the vanishing of $s_\eta$.

\begin{prop}There exist a section 
$$s_\eta\in H^0\Big(S, q_\eta \big(K_{S}+ \wh B _{\eta|S}\big)\Big)$$
whose zero set contains the divisor 
$$q_\eta\Big(\sum_{j\in J\setminus I}\rho^{\infty, j}_\eta[Y_{j|S}]+  
\sum_{j\in I}\nu^{j}_{\eta}[Y_{j|S}]\Big)$$ 
for all $0<\eta\ll 1$.

\end{prop}

\medskip
\begin{pfof}{[3.8]} We first remark that we have
$$q_\eta \big(K_{S}+ \wh B _{\eta|S}\big)= K_S+ (q_\eta-1) \big(K_{S}+ 
\wh B _{\eta|S}\big)+ \wh B_{\eta|S} ;$$
in order to use the classical vanishing theorems, we have to endow the bundle
$$(q_\eta-1) \big(K_{S}+ 
\wh B _{\eta|S}\big)+ \wh B_{\eta|S}$$
with an appropriate metric.

We first consider the $\Q$--bundle $\wh B _{\eta}$~; we will construct a metric on
it induced by the decomposition 
$$\wh B _\eta= \wh B+ \big(\wh B _\eta-\wh B\big).$$
The second term above admits a smooth representative whose local weights are 
bounded by $\displaystyle {{\eta}\over {q_\eta}}$ in ${\mathcal C}^\infty$ norm, by the approximation
relation ${\rm A.3}$. As for the first one, we recall that we have
$$\wh B= \sum_{j\in J}\nu^{j}Y_j+ \wh \Lambda_B~~; \leqno(26)$$
where the (1,1)-form $\wh \Lambda_B$ has the positivity properties in 3.5, 2.

\smallskip 
\noindent Now, the first metric we consider on $\wh B_{\eta|S}$ is defined such that its curvature current is equal to
$$\sum_{j\in I}\max\big(\nu^{j}, \nu^{j}_{\eta}\big)Y_{j|S}+ 
\sum_{j\in J\setminus I }\nu^{j}Y_{j|S}+ \wh \Lambda_{B|S}+ \Xi(\eta)_{|S}
\leqno(27)$$
where $\Xi(\eta)$ is a non-singular $(1,1)$--form on $\wh X$ in the class of
the current 
$$\sum_{j\in I}\Big(\nu^{(j)}-\max\big(\nu^{(j)}, \nu^{(j)}_{\eta}\big)\Big)[Y_j]$$
plus $\wh B_\eta-\wh B$ ~; 
we can assume that it is greater than $\displaystyle -C{{\eta}\over {q_\eta}}$,
where the constant $C$ above {\sl is independent of $\eta$}.

\medskip
\noindent The smooth term $\wh \Lambda_{B}$ is semi-positive on $\wh X$ and strictly positive at the generic point of $S$ : thanks to this positivity properties we can find a representative of the class
 $\{ \wh \Lambda_{B}\}$ which dominates a K\"ahler metric. In general we cannot avoid that this representative acquire some singularities.
 However, in the present context 
 we will show that there exist current in the above class which is ``restrictable" to $S$. 
  
Indeed, we consider the exceptional divisors $(E_j)$ of the map
$\mu_1$ (see the proposition 3.5)~~; the hypersurface $S$ do not belong to this set, and 
then the class 
$$\wh \Lambda_{B}-\sum_{j}\varepsilon^jE_j$$
is ample on $\wh X$, for some positive reals $\varepsilon^j$. Once a set of such parameters is chosen, we fix a K\"ahler form
$$\Omega\in \{\wh \Lambda_{B}-\sum_{j}\varepsilon^jE_j\}$$
and for each $\delta\in [0, 1]$ we define 
$$\wh \Lambda_{B, {\delta}}~:= (1-\delta)\wh \Lambda_{B}+ \delta \big(\Omega+
\sum_j\varepsilon^jE_j\big).\leqno(28)$$
For each $\eta > 0$, there exist $\delta> 0$ such that the differential form
$\delta\Omega+ ~\Xi(\eta)$
is positive defined. For example, we can take 
$$\delta~:= C\displaystyle {{\eta}\over {q_\eta}}\leqno(29)$$ where the constant $C> 0$ does not depends on $\eta$. 

\noindent With the choice of several parameters as indicated above, the current 
$$\wh \Lambda_{B, {\delta}}+ \Xi(\eta)$$
dominates a K\"ahler metric, and since $\wh \Lambda_{B, {\delta}}$ is in the same cohomology class as $\wh \Lambda_{B}$, we have 
$$\wh B\equiv\sum_{j\in I}\max\big(\nu^{j}, \nu^{j}_{\eta}\big)Y_{j}+ 
\sum_{j\in J\setminus I }\nu^{j}Y_{j}+ \wh \Lambda_{B, {\delta}}+ \Xi(\eta).
\leqno(30)$$
We remark that the current in the expression above admits a well-defined restriction to
$S$ ; moreover, the additional singularities of the restriction (induced by $\wh \Lambda_{B, {\delta}}$) are of order $\displaystyle C\displaystyle {{\eta}\over {q_\eta}}$, thus il will clearly be klt as soon as $\eta \ll1$. The current in the expression (30) induce a metric on $\wh B_{\eta|S}$.

\smallskip

Next, we define a singular metric 
on the bundle $(q_\eta-1) \big(K_{S}+ \wh B _{\eta|S}\big)$
whose curvature form is equal to $(q_\eta-1)\wh T_{S,\eta}$ and
we denote by $h_\eta$ the resulting metric on the bundle
$$(q_\eta-1) \big(K_{S}+ \wh B _{\eta|S}\big)+ B_{\eta|S}.$$
\smallskip

\noindent The divisor $q_\eta\wh T_{S,\eta}$ corresponds to the current of integration along the zero set of the section $u_\eta$ of the bundle 
$$q_\eta \big(K_{S}+ \wh B _{\eta|S}\big)+ \rho$$
where $\rho$ is a topologically trivial line bundle on $S$. 
\smallskip

By the Kawamata-Viehweg-Nadel vanishing theorem (cf. [\ref{kawa1}], [\ref{eckart}], [\ref{nad}]) we have
$$H^j(S, q_\eta \big(K_{S}+ \wh B _\eta\big)\otimes \cI\big(h_\eta\big)\Big)= 0$$
for all $j\geq 1$, 
and the same is true for the bundle $q_\eta \big(K_{S}+ \wh B _\eta\big)+ \rho$, since
$\rho $ carries a metric with zero curvature.
Moreover, the section $u_\eta$ belong to the multiplier ideal of the metric 
$h_\eta$ above, as soon as $\eta$ is small enough, because the multiplier ideal of the metric on the bundle $\wh B_{\eta|S}$ will be trivial. Since 
the Euler characteristic of the two bundles is the same, we infer that
$$H^0\Big(S, q_\eta \big(K_{S}+ \wh B _\eta\big)\otimes \cI\big(h_\eta\big)\Big)\neq 0$$
We denote by $s_\eta$ any non-zero element in the group above~~; we show now that its
zero set satisfy the requirements in the lemma. Indeed, locally at any point of $x\in S$
we have
$$\int_{(S, x)}{{|f_s|^2}\over {\prod _{j\in J\setminus I}|f_j|^{2\rho^{\infty, j}_\eta(q_\eta-1)+ 2\wt\nu^{(j)}_{\eta}}
\prod_{j\in I}|f_j|^{2\nu^{j}_{\eta}(q_\eta-1)+ 2\wt\nu^{j}_{\eta}}}}
d\lambda<\infty $$
where $\wt\nu^{j}_{\eta}~:= \nu^{j}$ if $j\in J\setminus I$ and $\wt\nu^{j}_{\eta}~:= 
\max\{\nu^{j}_{\eta}, \nu^{j}\}$ if $j\in I$ ~; we denote by $f_s$ the local expression of the section $s_\eta$, and we denote by $f_j$ the local equation of $Y_j\cap S$.

But the we have 
$$\int_{(S, x)}{{|f_s|^2}\over {\prod _{j\in J\setminus I}|f_j|^{2\rho^{\infty, j}_\eta q_\eta}
\prod_{j\in I}|f_j|^{2\nu^{j}_{\eta}q_\eta}}}d\lambda<\infty$$
for all $\eta\ll 1$ (by the definition of the set $I$ and the construction of the metric on 
$\wh B_{\eta|S}$). Therefore, the lemma is proved.\hfill\qed

\end{pfof}

\smallskip
\begin{rem}{\rm Concerning the construction and  the properties 
of $\wh \Lambda_{B, \delta}$, we recall the very nice result in [\ref{ELMNP}], stating that 
{\sl if $D$ is an $\R$-divisor which is nef and big, then its associated augmented base locus can be determined numerically.} \\
}\end{rem}

\smallskip
\begin{rem}{\rm As one can easily see, the divisor we are interested  
in the previous proposition 3.7 is given by
$$E_\eta:= \sum_{j\in J\setminus I}\rho^{\infty, j}_\eta[Y_{j|S}]+  
\sum_{j\in I}\nu^{j}_{\eta}[Y_{j|S}].$$
The crucial fact about it is that it is {\sl smaller} than the singularities of the metric we construct for $\wh B_\eta$ ; this is 
the reason why we can infer that the section $s_\eta$ above vanishes on 
$q_\eta E_\eta$ --and not just on the round down of the divisor 
$(q_\eta- 1)E_\eta$--, see [\ref{mp2}], page 42 for some comments about this issue.
\\
}\end{rem}

%%%%%%%%%%%%%%%%%%%%%%%%%%%%%%%%%%%%%%%%%%%%%%%%%%%%%%%%%%%%%%%%%%%%%%%%%%%%%%
%\bigskip

\subsection*{The method of Siu.\\\\}
%\medskip

\noindent We have arrived at the 
last step in our proof ~: {\sl for all $0< \eta\ll 1$, the section $s_\eta$ admit 
an extension on $\wh X$}. Once this is done, we just use the point A.4 of the approximation lemma 3.8, in order to infer the  existence of a $\R$--section of the bundle 
$K_{\wh X}+ S+\wh B$, and then the relation $(3)$ of 3.5 to conclude. 

The extension of the section $s_\eta$ will be obtained by using the {\sl invariance of plurigenera} techniques, thus in the first paragraph of the current subsection, we will highlight some of the properties of the $\Q$-divisors $\wh B_\eta$ 
constructed above. \\

%\medskip 
\subsection*{Uniformity properties of $(K_{\wh X}+ S+ \wh B_\eta)_{\eta > 0}$.\\\\}

%\medskip
\noindent We list below the pertinent facts which will ultimately enable us to perform the extension of $(s_\eta)$~; {\sl the constant $C$ which appear in the next 
statement is independent of $\eta$}.
\begin{enumerate}

\item [{($\bf U_1$)}] The section $s_\eta\in H^0\big(S, q_\eta(K_S+ \wh B_\eta)\big)$
vanishes along the divisor
$$q_\eta\Big(\sum_{j\in J\setminus I}\rho^{\infty, j}_\eta[Y_{j|S}]+  
\sum_{j\in I}\nu^{j}_{\eta}[Y_{j|S}]\Big)$$ 
for all $0<\eta\ll 1$ \hfill\qed

\smallskip
%\itemindent -7pt

\item [$\bf (U_2)$] There exist a closed (1,1)--current $\Theta_\eta\in \{K_{\wh X}+ S+\wh B_\eta\}$ such that ~:
\smallskip

\begin{enumerate} 
\smallskip

\item [$\bf ({2.1})$] It is greater than $\displaystyle -C{{\eta}\over {q_\eta}}\omega$~~;
\smallskip

\item [$\bf ({2.2})$] Its restriction to $S$ is well defined, and we have
$$\Theta_{\eta|S}= \sum _{j\in J}\theta^j_\eta[Y_{j|S}]+ R_{\eta, S}.$$
Moreover, the support of the divisor part of $R_{\eta, S}$ is disjoint from the set
$(Y_{j|S})$ and $\theta^j_\eta\leq \rho^{\infty, j}_\eta+ C{{\eta}\over {q_\eta}}$.
\smallskip

\end{enumerate}
\item [{($\bf U_3$)}] The bundle $\wh B_{\eta}$ can be endowed with a metric whose curvature current is given by 
$$\sum_{j\in J}\nu^{j}_\eta[Y_j]+ 
\wh \Lambda_{B, \eta}+ \Xi(\eta)$$
where the hypersurfaces $Y_j$ above verify 
$Y_j\cap Y_i=\emptyset$, if $i\neq j$ and moreover we have :

\begin{enumerate} 
\smallskip

\item [$\bf ({3.1})$] The current $\wh \Lambda_{B, \eta}+ \Xi(\eta)$ dominates a K\"ahler metric ;
\smallskip

\item [$\bf ({3.2})$] The restriction $\displaystyle {\wh \Lambda_{B, \eta}+ \Xi(\eta)}_{|S}$is well defined, and if we denote by $\nu_\eta$ the maximal multiplicity of the above restriction
then we have
$$q_\eta \nu_\eta\leq {C}{\eta}.$$
\end{enumerate}

\end{enumerate}
\bigskip

\noindent The property $\bf (U_1)$ is a simple recapitulation of facts which were completely proved during the previous paragraphs. 
\smallskip

\noindent The family of currents in 
$\bf (U_2)$ can be easily obtained thanks to the proposition 3.5, by the definition of the 
quantities $\rho^{\infty, j}$ and their approximations. 
\smallskip

\noindent Finally, the construction of the metric on $\wh B_\eta$ as above is done 
precisely as in the previous paragraph, except that instead of taking the coefficients
$\max(\nu^j, \nu^j_\eta)$, we simply consider $\nu^j_\eta$. The negativity of the error term is the same (i.e. $C\eta/q_\eta$).\hfill\qed

\smallskip

\noindent Let us introduce the next notations ~:
\smallskip

\begin{enumerate}
%\itemindent -25pt
\item[$\bullet$] $\Delta_1~:= \sum_{j\in J\setminus I}\nu^{j}_\eta[Y_{j}]$. It 
is an effective and klt $\Q$--bundle ; notice that the 
multiple $q_\eta \nu^{j}_\eta$ is a positive integer strictly smaller than $q_\eta$, for each $j\in J\setminus I$~~;
\smallskip

\item [$\bullet$] $\Delta_2~:= 
\sum_{j\in I}\nu^{j}_\eta[Y_{j}] + \wh \Lambda_{B,\eta}+ \Xi(\eta) $. It is equally a effective and klt $\Q$-bundle such that $q_\eta \Delta_2$ is integral. 

\end{enumerate}

\medskip
\noindent Precisely as in [\ref{Dem2}], [\ref{EinP}], [\ref{BoB}], there exist a decomposition
$$q_\eta\Delta_1= L_1+...+L_{q_\eta-1}$$
such that for each $m= 1,..., q_\eta-1$, we have 
$$L_m~:= \sum_{j\in I_m\subset J\setminus I}Y_j.$$

We denote by $\displaystyle L_{q_\eta}~:= q_\eta\Delta_2$ and
$$L^{(p)}~:= p(K_X+ S)+L_1+...+ L_p\leqno(31)$$
where $p= 1,..., q_\eta$. By convention, $L^{(0)}$ is the trivial bundle.

\medskip

\noindent We remark that it is possible to find an ample bundle $(A, h_A)$ {\sl independent of $\eta$} whose curvature form is positive enough such that the next relations hold.

\begin{enumerate}

\item [{$(\dagger)$}] For each $0 \leq p\leq q_\eta-1$, the bundle
$L^{(p)}+ q_\eta A$ is generated by its global sections, which we denote by
$(s^{(p)}_j)$.

\item [{$(\dagger^2)$}] Any section of the bundle $L^{(q_\eta)}+
q_\eta A_{\vert S}$ admits an extension to 
$\wh X$.

\item [{$(\dagger^3)$}] We endow the bundle corresponding to $(Y_j)_{j\in J}$
with a non-singular metric, and we denote by $\wt \varphi_m$ the induced 
metric on $L_m$. 
Then for each $m= 1,..., q_\eta$, the functions
$$\wt \varphi_{L_m}+ 1/3\varphi_A$$
are strictly psh.

\item [{$(\dagger^4)$}] For any $\eta > 0$ we have
$$\Theta_\eta\geq -{{\eta}\over {q_\eta}}\Theta_A.  \leqno (32)$$
  \hfill\qed
\end{enumerate}

\medskip
\noindent Under the numerous assumptions/normalizations above, we formulate the next statement.

\begin{claim} There exist a constant $C> 0$ independent of $\eta$ 
such that the section 
$$s_\eta^{\otimes k}\otimes s^{(p)}_j\in H^0\bigl(S, L^{(p)}+ kL^{(q_\eta)}+ q_\eta A_{\vert S}\bigr)$$
extend to $\wh X$, for each $p= 0,..., q_\eta-1$, $j= 1,..., N_p$ and $k\in \Z_+$ such that 
$$k\frac{\eta}{q_\eta}\leq C.$$
\hfill\qed
\end{claim}

\medskip 
\noindent The statement above 
can be seen as a natural generalization of the usual invariance of plurigenera
setting (see [\ref{Clodo}], [\ref{Dem2}], [\ref{EinP}], [\ref{Kaw}], [\ref{kim}], [\ref{mp1}], [\ref{siu1}], [\ref{taka}], [\ref{dror}])~; in substance, we are about to say that the more general hypothesis 
we are forced to consider induce an {\sl effective limitation} of the number of iterations
we are allowed to perform. \\

%%%%%%%%%%%%%%%%%%%%%%%%%%%%%%%%%%%%%%%%%

\subsection*{ Proof of the claim 3.10.\\\\}
%\medskip

\noindent To start with, we recall the following very useful integrability criteria
(see e.g. [\ref{HM}]).

\begin{lem}Let $D$ be an effective $\R$-divisor on a manifold $S$. We consider the non-singular hypersurfaces $Y_j\subset S$ for 
$j=1,...,N$ such that $Y_j\cap Y_i= \emptyset$ if $i\neq j$, and such that 
the support of $D$ is disjoint from the set $(Y_j)$.
Then there exist a constant $\varepsilon_0~:= \varepsilon_0(\{D\}, C)$ depending only on 
the cohomology class of the divisor $D$ 
such that 
for all positive real numbers $\delta\in ]0, 1]$ and
$\varepsilon \leq \varepsilon_0$ we have
$$\int_{(S, s)}\frac{d\lambda}{|f_D|^{2\varepsilon}\prod_j|f_j|^{2(1-\delta)}}<
\infty $$ 
for all $s\in S$.\hfill\qed
\end{lem}

\noindent In the statement above, we denote by $f_j, f_D$ the local equations of $Y_j$, respectively $D$ near $s\in S$ (with the usual abuse of notation). 
\medskip

\noindent We will equally need the following version of the
Ohsawa-Takegoshi theorem (see [\ref{Dem1}], [\ref{mcv}], [\ref{ot}], [\ref{siu1}])~~; it will be our main technical tool in the proof of the claim.
\medskip

\begin{thm} [\ref{mcv}] Let $\wh X$ be a projective
$n$-dimensional 
manifold, and let $S\subset \wh X$ be a non-singular hypersurface.
Let $F$ be a line bundle,
equipped with a metric $h_F$. We assume that ~:  

\begin{enumerate}
\smallskip

\item [{(a)}] The curvature current $\displaystyle {{\sqrt {-1}}\over {2\pi}}\Theta_F$ 
is greater than a K\"ahler metric on $\wh X$~~;
\smallskip

\item [{(b)}] The restriction
of the metric $h_F$ on $S$ is well defined.

\end{enumerate}
\smallskip

\noindent Then every section $u\in H^0\bigl(S, (K_{\wh X} + S+
F_{\vert S})\otimes {\mathcal I}(h_{F|S})\bigr)$ admits an extension
$U$ to $\wh X$. \hfill\qed
\end{thm}
\medskip

\noindent We will use inductively the extension theorem 3.12, in order to derive a lower bound for the power $k$ we can afford in the invariance of plurigenera algorithm, under the conditions 
$\bf (U_j)_{\rm 1\leq j\leq 3}$~~; the first steps are as follows.
\medskip

\emph{Step 1.}   For each $j= 1,...,N_0$, the section 
$s_\eta\otimes s^{(0)}_j\in H^0\bigl({S}, L^{(q_\eta)}+ q_\eta A_{\vert { S}}\bigr)$
admits an extension $U^{(q_\eta)}_j\in H^0\bigl( X,  L^{(q_\eta)}+ q_\eta A\bigr)$,
by the property $\dagger\dagger$.

\emph{Step 2.}  We use the sections $(U^{(q_\eta)}_j)$ to construct a metric $\varphi^{(q_\eta)}$
on the bundle $L^{(q_\eta)}+ q_\eta A$.

\emph{Step 3.}  Let us consider the section 
$s_\eta \otimes s^{(1)}_j\in H^0\bigl({ S}, L^{(1)}+ L^{(q_\eta)}+ q_\eta A_{\vert { S}}\bigr)$. 
We remark that the bundle
$$L^{(1)}+ L^{(q_\eta)}+ q_\eta A = K_{\wh X}+ {S}+ L_1+ L^{(q_\eta)}+ q_\eta A$$
can be written as $K_{\wh X}+ { S}+ F$ where
$$F~:=  L_1+ L^{(q_\eta)}+ q_\eta A $$
thus we have to construct a metric on $F$ which satisfy the 
curvature and integrability assumptions in the Ohsawa-Takegoshi-type theorem
above.

Let $\delta, \varepsilon$ be positive real numbers~~; we endow the bundle 
$F$ with the metric given by
$$\varphi^{(q_\eta)}_{\delta, \varepsilon}~:= (1-\delta)\varphi_{L_1}+ \delta\wt \varphi_{L_1}+ 
(1-\varepsilon)\varphi^{(q_\eta)}+ \varepsilon q_\eta(\varphi_A+ \varphi_{\Theta_\eta})
\leqno(33)$$
where the metric $\wt \varphi_{L_1}$ is smooth (no curvature requirements) and  
$\varphi_{L_1}$ is the weight of the singular metric induced by the divisors 
$\displaystyle (Y_j)_{j\in I_1}$. We denote by $\displaystyle \varphi_{\Theta_\eta}$
the local weight of the current $\Theta_\eta$ ; it induces a metric on the corresponding 
$\Q$-bundle $K_{\wh X}+ S+ \wh B_\eta$, which is used above.

We remark that the curvature conditions in the extension theorem will be fulfilled if
$$\delta< \varepsilon q_\eta$$
provided that $\eta\ll 1$ : by the relations $(\dagger^3)$ and $(\dagger^4)$
 the negativity of the curvature induced by the term $\delta\wt \varphi_{L_1}$ will be absorbed by $A$.

\noindent Next we claim that the sections $s_\eta\otimes s^{(1)}_j$ are integrable with respect to the metric defined in (33), provided that the parameters $\varepsilon, \delta$ are chosen in an appropriate manner.
Indeed, we have to prove that 
$$\int_{ S}{{\vert s_\eta\otimes s^{(1)}_j\vert^2}\over {(\sum_r \vert s_\eta 
\otimes s^{(0)}_r\vert^2)^{1-\varepsilon}}}
\exp \big(-(1-\delta)\varphi_{L_1}- \varepsilon q_\eta\varphi_{\Theta_\eta}\big)dV< \infty ~;$$
since the sections $(s^{(0)}_r)$ have no common zeroes, it is enough to show that 
$$\int_{ S}\vert s_\eta \vert ^{2\varepsilon}
\exp \big(-(1-\delta)\varphi_{L_1}- \varepsilon q_\eta\varphi_{\Theta_\eta}\big)dV< \infty$$
(we have abusively removed the smooth weights in the above expressions, to simplify the writing). 

Now the property $(\bf U_1)$ concerning the zero set of $s_\eta$ is used ~: the above integral is convergent, provided that we have
$$\int_{ S}\exp \big(-(1-\delta)\varphi_{L_1}- 
\varepsilon q_\eta(\varphi_{\Theta_\eta}- 
\sum_{j\in J\setminus I}\rho^{\infty, j}_\eta \varphi_{Y_j}-  
\sum_{j\in I}\nu^{j}_\eta\varphi_{Y_j})\big)dV< \infty.$$
In order to conclude the convergence of the above integral,
we would like to apply the integrability lemma 3.11 ; therefore, we have to estimate 
the coefficients of the common part of the support of $L_{1|S}$ and 
$$\Theta_\eta- \sum_{j\in J\setminus I}\rho^{\infty, j}_\eta [Y_j]-  
\sum_{j\in I}\nu^{j}_\eta[Y_j]\leqno (34)$$
restricted to $S$. For any 
$j\in J\setminus I$, the coefficient associated to the divisor $Y_{j|S}$ in the expression above is
equal to 
$$\theta_\eta^j- \rho^{\infty, j}_\eta\leqno (35)$$
and by the property $\bf U_2$, the difference above is smaller than 
$\displaystyle C\frac {\eta}{q_\eta}$.
The singular part corresponding to $j\in J\setminus I$ in the expression 
$(34)$ will be 
incorporated into the $\displaystyle (1-\delta)\varphi_{L_1}$, thus we have to impose
the relation
$$1-\delta+ q_\eta\varepsilon C\frac {\eta}{q_\eta} < 1.$$

In conclusion, the positivity and integrability conditions will be satisfied provided that
$$C\eta\varepsilon  <\delta < \varepsilon q_\eta\leq 
\varepsilon_0\leqno(36)$$
We can clearly choose the parameters $\delta, \varepsilon$ such that (36) is verified.
\smallskip

\emph{Step 4.}  We apply the extension theorem and we get
$U^{(q_\eta+ 1)}_j$, whose
restriction on ${ S}$ is precisely
$s_\eta\otimes s^{(1)}_j$.\hfill \qed
\medskip

\vskip 5pt The claim will be obtained by iterating
the
procedure (1)-(4) several times, and estimating carefully the influence of the negativity of
$\Theta_\eta$ on this process. Indeed, assume that we already have the set of global sections 
$$U^{(kq_\eta+p)}_j\in H^0\bigl(\wh X, L^{(p)}+ kL^{(q_\eta)}+ q_\eta A\bigr)$$
which extend $s_\eta^{\otimes k}\otimes s^{(p)}_j$. They induce a metric on the above bundle, denoted by $\varphi^{(kq_\eta+p)}$. 

\noindent If $p< q_\eta-1$, then we define the family of 
sections 
$$s_\eta^{\otimes k}\otimes s^{(p+1)}_j\in H^0({ S}, L^{(p+1)}+ kL^{(q_\eta)}+ 
q_\eta A_{|{ S}})$$
on ${ S}$. 
As in the step (3) above we remark that we have
$$L^{(p+1)}= K_{\wh X}+ { S}+ L_{p+1}+ L^{(p)}$$
thus according to the extension result 3.12, we have to exhibit a metric
on the bundle 
$$F~:=  L_{p+1}+ L^{(p)}+kL^{(q_\eta)}+ q_\eta A$$
for which the curvature conditions are satisfied, and such that the family of sections above are $L^2$ with respect to it.
We define
$$\varphi^{(kq_\eta+p+1)}_{\delta, \varepsilon}~:= (1-\delta)\varphi_{L_{p+1}}+ \delta\wt \varphi_{L_{p+1}}+ 
(1-\varepsilon)\varphi^{(kq_\eta+ p)}+ \varepsilon q_\eta\big(k\varphi_{\Theta_\eta}+  \varphi_A+ {{1}\over {q_\eta}}\wt \varphi_{L^{(p)}}\big)\leqno(37)$$
and we check now the conditions that the parameters $\varepsilon, \delta$ have 
to satisfy.
\medskip

We have to absorb the negativity in the smooth curvature terms in (37), 
and the one from $\Theta_\eta$. The Hessian of the term
$$1/3\varphi_A+ {{1}\over {q_\eta}}\wt \varphi_{L^{(p)}}$$
is assumed to be positive by $\dagger^3$, but we also have a {\sl huge} negative contribution 
$$-Ck\frac {\eta}{q_\eta}\Theta_A$$
induced by the current $\Theta_\eta$. However, we remark that we can assume that we have 
$$Ck\frac {\eta}{q_\eta}< 1/3\leqno(38)$$
since this is {\sl precisely} the range of $k$ for which we want to establish the claim.
Then the curvature of the metric defined in (37) will be positive, provided that 
$$\delta< \varepsilon q_\eta$$
again by $(\dagger ^3)$. 

\medskip 
\noindent 
Let us check next the $L^2$ condition~~; we have to show that 
the integral below in convergent 

$$\int_{ S}{{|s^{\otimes k}\otimes s^{(p+1)}_j|^2}\over 
{(\sum_r \vert s^{\otimes k}\otimes s^{(p)}_r\vert^2)^{1-\varepsilon}}}
\exp \big(-(1-\delta)\varphi_{L_{p+1}}- kq_\eta 
\varepsilon\varphi_{\Theta_\eta}\big)dV.$$
This is equivalent with 
$$\int_{ S}|s|^{2\varepsilon k}
\exp \big(-(1-\delta)\varphi_{L_{p+1}}- kq_\eta 
\varepsilon\varphi_{\Theta_\eta}\big)dV< \infty.$$
In order to show the above inequality, we use the same trick as before ~: the vanishing set of the section 
$s_\eta$ as in $\bf (U_1)$ will allow us to apply the integrability lemma--the computations 
are strictly identical with those discussed in the point 3) above, but we give here some details. 

By the vanishing properties of the section $s_\eta$, the finiteness of the previous integral will be implied by the inequality
$$\int_{ S}\exp \big(-(1-\delta)\varphi_{L_{p+1}}- 
k\varepsilon q_\eta(\varphi_{\Theta_\eta}- 
\sum_{j\in J\setminus I}\rho^{\infty, j}_\eta \varphi_{Y_j}-  
\sum_{j\in I}\nu^{j}_\eta\varphi_{Y_j})\big)dV< \infty.$$

In the first place, we have to keep the poles of $k\varepsilon q_\eta \Theta_\eta$ ``small" in the expression of the metric (37), thus we impose 
$$k\varepsilon q_\eta\leq \varepsilon_0.$$
 \smallskip
The hypothesis in the integrability lemma will be satisfied provided that 
$$1-\delta+ \varepsilon kq_\eta C\frac {\eta}{q_\eta} < 1$$
(this is the contribution of the common part of $\Supp L_{p+1}$ and $\Theta_\eta$).
Combined with the previous relations, the conditions for the parameters become
$$ C\varepsilon k\eta < \delta< \varepsilon q_\eta< \varepsilon_0/k.\leqno(39)$$
Again we see that the inequalities above {\sl are compatible} if $k$ satisfy the inequality
$$Ck\eta <q_\eta$$
which is precisely what the claim (3.10) states.
\medskip

In conclusion,  we can choose the parameters
$\varepsilon, \delta$ so that the integrability/positivity conditions in the extension theorem are verified ~; for example, we can take
\smallskip

\noindent $\bullet$ $\displaystyle \varepsilon~:= {{\varepsilon_0}\over {2kq_\eta}}$

\noindent and 
\smallskip

\noindent $\bullet$ $\displaystyle \delta~:= \big(1+ kC\frac {\eta}{q_\eta}\big){{\varepsilon_0}\over
{4 k }}$.\hfill\qed

\medskip
\noindent Finally, let us indicate how to perform the induction step if $p= q_\eta-1$~~:
we consider the family of 
sections 
$$s_\eta^{k+1}\otimes s^{(0)}_j\in H^0({ S}, (k+1)L^{(q_\eta)}+ q_\eta A_{|{ S}}),$$ 
In the case under consideration, we have to exhibit a metric
on the bundle 
$$L_{q_\eta}+ L^{(q_\eta-1)}+kL^{(q_\eta)}+  q_\eta A~~;$$
however, this is easier than before, since we can simply take 
$$\varphi^{q_\eta(k+1)}~:= 
 q_\eta \varphi_{ \Delta_2}+ 
\varphi^{(kq_\eta+q_\eta-1)}\leqno(40)$$
where the metric on $ \Delta_2$ is induced by its expression in the preceding subsection.
With this choice, the curvature conditions are satisfied~~; as for the $L^2$ ones, we remark that we have
$$ \int_{ S}|s_\eta^{k+1}\otimes s^{(0)}_j|^2
\exp \big(-\varphi^{q_\eta(k+1)}\big)dV< C\int_{ S}|s_\eta\otimes s^{(0)}_j|^2
\exp \big(-q_\eta\varphi_{\Delta_2}\big)dV ~;$$
moreover, by the vanishing of $s_\eta$ along the divisor 
$$q_\eta\big(\sum_{j\in I}\nu^{j}_{\eta}[Y_{j|S}]\big),$$ 
the right hand side term of the inequality above is dominated by 
$$C\int_{ S} 
\exp \big(-q_\eta\varphi_{\wh \Lambda_{B, \eta}}\big)dV
$$
where the last integral is convergent because of the fact that
$q_\eta\nu< 1$, see $\bf (U_3)$.
The proof of the extension claim is therefore finished. \\

\hfill \qed

%%%%%%%%%%%%%%%%%%%%%%%%%%%%%%%%%%%%%%%%%

%%%%%%%%%%%%%%%%%%%%%%%%%%%%%%%%%%%%%%%%%
%\bigskip

\subsection*{End of the proof.\\\\}

%\medskip

\noindent We show next that the sections $s_\eta$ can be lifted to $\wh X$
as soon as $\eta$ is small enough, by using the claim 3.10.    

Indeed, we consider the extensions $U^{(kq_\eta)}_j$ of the sections 
$s_\eta ^{\otimes k}\otimes s^{(0)}_j$~~; they can be used to define a metric on the bundle
$$kq_\eta(K_{\wh X}+S+ \wh B_\eta)+ q_\eta A$$
whose $kq_\eta^{\rm th}$ root it is defined to be 
$h^{(\eta)}_k$. 

We write the bundle we are interested in i.e. 
$q_\eta(K_{\wh X}+S+ \wh B_\eta)$ as an adjoint bundle, as follows

$$ q_\eta(K_{\wh X}+S+ \wh B_\eta)=  K_{\wh X}+ S+ (q_\eta- 1)(K_{\wh X}+S+ \wh B_\eta)+ \wh B_\eta$$
and this last expression equals
$$ K_{\wh X}+ S+ (q_\eta- 1)\big(K_{\wh X}+S+ \wh B_\eta+ 1/kA\big)+ \wh B_\eta- \frac{q_\eta-1}{k}A 
$$

\medskip 
\noindent Given the extension theorem 3.12, we need to construct a metric on the bundle
$$(q_\eta- 1)\big(K_{\wh X}+S+ \wh B_\eta+ 1/kA\big)+ \wh B_\eta- {{q_\eta-1}\over 
{k}}A. $$
On the first factor of the above expression we will use 
$(q_\eta- 1)\varphi_k^{(\eta)}$ (that is to say, the $(q_\eta-1)^{\rm th}$ power of the metric given by
$h_k^{(\eta)}$). 

We endow the bundle $\wh B_\eta$ with a metric whose curvature is given by the expression 
$$\sum_{j\in J\setminus I}\nu^{j}_\eta[Y_{j}]+  
\sum_{j\in I}\nu^{j}_\eta[Y_{j}]+
\wh \Lambda_{B, \delta}+ \Xi(\eta)~~;$$
here 
we take $\delta$ {\sl independent of $\eta$}, but small enough such that the restriction $\wh B_{\eta|S}$ is still klt.
Finally, we multiply with the ${{q_\eta-1}\over 
{k}}$ times $h_A^{-1}$. 

\smallskip 

By the claim 3.10, we 
are free to choose $k$ e.g. such that $k= q_\eta\big[\eta^{-1/2}\big]$
(where $[x]$ denotes the integer part of the real $x$). Then the metric above is not identically $\infty$ when restricted to $S$, and its 
curvature will be 
strongly positive as soon as $\eta\ll 1$. Indeed, the curvature of $\wh B_\eta$ is greater 
than a  K\"ahler metric on $\wh X$ {\sl which is independent of $\eta$}
because of the factor $\wh \Lambda_{B, \delta}$. 

Moreover, the $L^2$ conditions in the theorem 3.12 are satisfied, since 
the norm of the section $s_\eta$ with respect to the metric $q_\eta\varphi_k^{(\eta)}$
is {\sl pointwise bounded}, and by the choice of the metric on 
$\wh B_{\eta|S}$. 
\medskip
\noindent In conclusion, we obtain an extension of the section $s_\eta$, and the theorem 1.5 is completely proved.\\
\hfill\qed

%\bigskip

\subsection*{The relative case.\\\\}

%\medskip

\noindent We will explain along the next lines the nonvanishing result 1.5 in its
general form ; to the end, we first review the notion of {\sl relative bigness } from metric point of view.

\smallskip

\noindent Let $p: X\to Z$ be a projective map and let $B$ be a $\R$-divisor on $X$. The pair $(X, B)$ is klt by hypothesis, so we can assume that $X$ is non-singular
and that 
$$B= \sum_{j= 1}^Na^jW_j\leqno(41) $$
where $0< a^j< 1$ and $(W_j)$ have normal crossings. Moreover, it is enough to prove 1.5 for non-singular manifolds $Z$ (since we can desingularize it if necessary, and modify further $X$).

The $\R$-divisor $B$ is equally $p$--big, thus there exist an ample bundle $A_X$, an effective divisor $E$
on $X$ and an ample divisor $A_Z$ on $Z$ 
such that 
$$B+ p^\star A_Z\equiv A_X+ E.\leqno(42) $$  
By a suitable linear combination of the objects give by the relations (41) and (42) above, we see that there exist a klt current 
$$\Theta_B\in \{B+ p^\star A_Z\}$$
which is greater than a K\"ahler metric. Thus modulo the
inverse image of a suitable bundle, the cohomology class of $B$ has precisely the same metric properties as in the absolute case.
\smallskip

The main technique we will use in order to settle 1.5 in full generality is the positivity properties of the 
twisted relative canonical bundles of projective surjections~; more precisely, the result we need is the following.

\begin{thm}[\ref{BoB}] Let $p: X\to Z$ be a projective surjection, and let $L\to X$ be a line bundle endowed with a metric $h_L$ with the following properties.

\begin{enumerate}

%\itemindent 5pt

\smallskip

\item The curvature current of $(L, h_L)$ is positive ;

\smallskip

\item There exist a generic point $z\in Z$, an integer $m$ and a non-zero section 
$u\in H^0(X_z, mK_{X_z}+ L)$
such that 
$$\displaystyle \int_{X_z}|u|^{\frac {2}{m}}\exp\big(-\frac{\varphi_L}{m}\big)d\lambda< \infty.$$

\end{enumerate}

\noindent Then the twisted relative bundle 
$mK_{X/Z}+ L$
is pseudo-effective, and it admits a positively curved metric $h_{X/Z}$ whose restriction to the generic fiber of $p$ is less singular than the metric induced by the
holomorphic sections who verify the $L^{2/m}$ condition in (2) above. \hfill\qed

\end{thm}

\medskip

\noindent Given this result, the end of the proof of 1.5 goes as follows. A point $z\in Z$ will be called {\sl very generic} if 
the restriction of $\Theta_B$
to the fiber $X_z$ dominates a K\"ahler metric and its singular part is klt, and moreover if the sections of all multiples of rational approximations of $K_X+ B$ restricted to $X_z$ do 
extend near $z$. We see that the set of very generic points of $z$
is the complement of a countable union of Zariski closed algebraic sets ; in particular, it is non-empty.

Let $z\in Z$ be a generic point. The adjoint $\R$--bundle $K_{X_z}+ B_{|X_z}$
is pseudo-effective, thus by the absolute case of 1.5 we obtain an effective $\R$--divisor
$$\Theta:= \sum_{j=1}^N\nu^jW_j$$
within the cohomology class of $K_{X_z}+ B_{|X_z}$. By diophantine approximation we obtain a family of $\Q$-bundles $(B_\eta)$ and a family of 
non-zero holomorphic sections 
$$u_\eta\in H^0\big(X_z, q_\eta(K_{X_z}+ B_{\eta|X_z})\big)$$
induced by the rational approximations of $\Theta$ (see 3.6, 3.8 above). 
\smallskip

With these datum, the theorem 3.17 provide the bundle 
$$q_\eta(K_{X/Z}+ B_{\eta})$$
with a positively curved metric $h_{X/Z}$, together 
with a crucial {\sl quantitative} information : the section $u_\eta$ is 
{\sl bounded} with respect to it.

The last step is yet another application of the Ohsawa-Takegoshi type theorem 3.12.
Indeed, we consider the bundle
$$q_\eta(K_{X}+ B_{\eta})+ p^\star A$$
where $A\to Z$ is a positive enough line bundle, such that $A-(q_\eta-1)K_Z$ is ample.
We have the decomposition
$$q_\eta(K_{X}+ B_{\eta})+ p^\star A= K_X+ (q_\eta-1)(K_{X/Z}+ B_{\eta})+ B_\eta
+ p^\star \big(A-(q_\eta-1)K_Z\big)$$
and we have to construct a metric on the bundle
$$F:= (q_\eta-1)(K_{X/Z}+ B_{\eta})+ B_\eta
+ p^\star \big(A-(q_\eta-1)K_Z\big)$$
with the curvature conditions as in 3.12. The first term in the sum above is 
endowed with the multiple $\displaystyle \frac{q_\eta-1}{q_\eta}$ of the metric $h_{X/Z}$.
The $\Q$-bundle $B_\eta$ is endowed with the metric given by $\Theta_B$ plus a smooth term corresponding to the difference $B_\eta-B$. Finally, the last term has a non-singular metric with positive curvature, thanks to the choice of $A$ ; one can see that with this choice, the curvature assumptions in 3.12 are satisfied. 

The klt properties of $B$ are inherited by $B_\eta$ ; thus we have
$$\int_{X_z}|u_\eta|^2\exp \Big(-\frac{q_\eta-1}{q_\eta}\varphi _{X/Z}-\varphi_{B_\eta}\Big)d\lambda\leq C\int_{X_z}\exp (-\varphi_{B_\eta})d\lambda< \infty.$$
In conclusion,
we can extend $u_\eta$ to the whole manifold $X$ by 3.12. The convexity argument in the lemma 3.6
ends the proof of the nonvanishing.\hfill\qed

\smallskip
\begin{rem}{\rm In fact, V. Lazic informed us that given the non-vanishing statement 1.5 in {\sl numerical setting}, he can infer the original 
non-vanishing statement in [\ref{BCHM}] (see [\ref{Laz}], as well as [\ref{Laz1}]). 
As a consequence, one can infer the relative version of 1.5 in the same 
way as in [\ref{BCHM}].
 \\
}\end{rem}

%%%%%
%%%%%%%%%%%%%%%%%%%%%%%%%%%%%%%%%%%%%%%%%%%%%%%%

\section{Proof of main resutls}
\begin{proof}(of Theorem \ref{main})
We proceed by induction on $d$. Suppose that the theorem holds in dimension $d-1$ and 
let $(X/Z,B)$ be a klt pair of dimension $d$ such that $B$ is big$/Z$. First assume that 
$K_X+B$ is pseudo-effective$/Z$. Then, by Theorem \ref{t-nonvanishing} in dimension $d$, 
$K_X+B$ is effective$/Z$. Theorem \ref{t-mmodels} then implies that $(X/Z,B)$ has a 
log minimal model. 

Now assume that $K_X+B$ is not pseudo-effective$/Z$ and let $A$ be a general 
ample$/Z$ $\Q$-divisor such that $K_X+B+A$ is klt and nef$/Z$. Let $t$ be the smallest number such that $K_X+B+tA$ is pseudo-effective$/Z$. By part (1), there is a log minimal model $(Y/Z,B_Y+tA_Y)$ for $(X/Z,B+tA)$. Run the LMMP$/Z$ on $K_Y+B_Y$ with scaling of $tA_Y$. 
By Theorem \ref{t-term}, we end up with a Mori fibre space for $(X/Z,B)$. 
\end{proof}

\begin{proof}(of Corollary \ref{c-flips})
Let $(X/Z,B)$ be a klt pair of dimension $d$ and $f\colon X\to Z'$  a
$(K_X+B)$-flipping contraction$/Z$. By (1) of Theorem \ref{main}, there is a
log minimal model $(Y/Z',B_Y)$ of $(X/Z',B)$.  By the base point free theorem, $(Y/Z',B_Y)$  has a log canonical model which gives the flip of $f$. 
\end{proof}

\begin{proof}(of Corollary \ref{c-fg})
If $K_X+B$ is not effective$/Z$, then the corollary trivially holds. So, assume otherwise.
By [\ref{FM}] there exist a klt pair $(S/Z,B_S)$ of dimension $\le \dim X$ with big$/Z$ $\Q$-divisor $B_S$, and $p\in \N$ such that locally over $Z$ we have 
$$
H^0(mp(K_X+B))\simeq H^0(mp(K_S+B_S))
$$ 
for any $m\in \N$. By Theorem \ref{main}, we may assume that $K_S+B_S$ is nef$/Z$. 
The result then follows as $K_S+B_S$ is semi-ample$/Z$ by the base point free theorem.
\end{proof}

%%%%%%%%%%%%%%%%%%%%%%%%%%%%%%%%%%%%%%%%%%%%%%%%%

\end{document}

\bibitem [15]&Demailly, J.-P.:& 

\bibitem [16]&&